\documentclass{gtart_a}
\pdfoutput=1

\usepackage{amscd}
\usepackage{pinlabel}

\title[Global medial structure]{The global medial structure of regions in $\mathbb{R}^3$}
\author{James Damon}
\givenname{James}
\surname{Damon}
\address{Department of Mathematics\\
University of North Carolina\\\newline
Chapel Hill, NC 27599-3250\\USA}
\email{jndamon@math.unc.edu}
\urladdr{}

\volumenumber{10}
\issuenumber{}
\publicationyear{2006}
\papernumber{54}
\lognumber{0709}
\startpage{2385}
\endpage{2429}

\doi{}
\MR{}
\Zbl{}

\arxivreference{} 


\received{2 February 2006}
\revised{}
\accepted{30 August 2006}
\published{15 December 2006}
\publishedonline{15 December 2006}
\proposed{Colin Rourke}
\seconded{Robion Kirby, Walter Neumann}
\corresponding{}
\editor{Liz}
\version{}

\AtBeginDocument{\let\bar\wbar\let\tilde\wtilde\let\hat\what}
\numberwithin{equation}{section}

\makeatletter
\def\cnewtheorem#1[#2]#3{\newtheorem{#1}{#3}[section]
\expandafter\let\csname c@#1\endcsname\c@Thm}

\theoremstyle{plain}
\newtheorem{Thm}{Theorem}[section]
\cnewtheorem{TitleThm}[Thm]{Fundamental Group Condition}
\cnewtheorem{Corollary}[Thm]{Corollary}
\cnewtheorem{Proposition}[Thm]{Proposition}
\cnewtheorem{Lemma}[Thm]{Lemma}
\cnewtheorem{Conjecture}[Thm]{Conjecture}
\theoremstyle{definition}
\cnewtheorem{Definition}[Thm]{Definition}
\cnewtheorem{Example}[Thm]{Example}
\newtheorem{TitleExample}{Algorithm for decomposing medial axis into irreducible 
components}

\cnewtheorem{Remark}[Thm]{Remark}

\makeatother  

\makeautorefname{TitleExample}{Algorithm}
\makeautorefname{TitleThm}{Condition}
\makeautorefname{Thm}{Theorem}
\makeautorefname{Corollary}{Corollary}
\makeautorefname{Lemma}{Lemma}
\makeautorefname{Proposition}{Proposition}

\newcommand{\coker}{{\rm coker}\,}

\def \iti {\textit{i}}
\def \gP {\Pi}
\def \gG {\Gamma}
\def \cY {\mathcal{Y}}
\def \cT {\mathcal{T}}
\def \cB {\mathcal{B}}
\def \gl {\lambda}
\def \gL {\Lambda}
\def \ga {\alpha}
\def \gb {\beta}
\def \gevar {\varepsilon}
\def \gD {\Delta}
\def \gd {\delta}

\let\tsty\textstyle

\begin{document}

\begin{asciiabstract}
For compact regions Omega in R^3 with generic smooth boundary B, we
consider geometric properties of Omega which lie midway between their
topology and geometry and can be summarized by the term ``geometric
complexity".  The ``geometric complexity" of Omega is captured by its
Blum medial axis M, which is a Whitney stratified set whose local
structure at each point is given by specific standard local types.

We classify the geometric complexity by giving a structure theorem for
the Blum medial axis M.  We do so by first giving an algorithm for
decomposing M using the local types into "irreducible components" and
then representing each medial component as obtained by attaching
surfaces with boundaries to $4$--valent graphs.  The two stages are
described by a two level extended graph structure.  The top level
describes a simplified form of the attaching of the irreducible medial
components to each other, and the second level extended graph
structure for each irreducible component specifies how to construct
the component.

We further use the data associated to the extended graph structures to
compute topological invariants of Omega such as the homology and
fundamental group in terms of the singular invariants of M defined
using the local standard types and the extended graph structures.
Using the classification, we characterize contractible regions in
terms of the extended graph structures and the associated data.
\end{asciiabstract}

\begin{htmlabstract}
<p class="noindent">
For compact regions &Omega; in <b>R</b><sup>3</sup> with generic smooth boundary
<b>B</b>, we consider geometric properties of &Omega; which lie
midway between their topology and geometry and can be summarized by
the term &ldquo;geometric complexity&rdquo;.  The &ldquo;geometric complexity&rdquo; of
&Omega; is captured by its Blum medial axis M, which is a Whitney
stratified set whose local structure at each point is given by
specific standard local types.
</p>
<p class="noindent">
We classify the geometric complexity by giving a structure theorem for
the Blum medial axis M.  We do so by first giving an algorithm for
decomposing M using the local types into &ldquo;irreducible components&rdquo;
and then representing each medial component as obtained by attaching
surfaces with boundaries to 4&ndash;valent graphs.  The two stages are
described by a two level extended graph structure.  The top level
describes a simplified form of the attaching of the irreducible medial
components to each other, and the second level extended graph
structure for each irreducible component specifies how to construct
the component.
</p>
<p class="noindent">
We further use the data associated to the extended graph structures to
express topological invariants of &Omega; such as the homology and
fundamental group in terms of the singular invariants of M defined
using the local standard types and the extended graph structures.
Using the classification, we characterize contractible regions in
terms of the extended graph structures and the associated data.
</p>
\end{htmlabstract}

\begin{abstract}
For compact regions $\Omega$ in $\mathbb{R}^3$ with generic smooth boundary
$\mathcal{B}$, we consider geometric properties of $\Omega$ which lie
midway between their topology and geometry and can be summarized by
the term ``geometric complexity''.  The ``geometric complexity'' of
$\Omega$ is captured by its Blum medial axis $M$, which is a Whitney
stratified set whose local structure at each point is given by
specific standard local types.

We classify the geometric complexity by giving a structure theorem for
the Blum medial axis $M$.  We do so by first giving an algorithm for
decomposing $M$ using the local types into ``irreducible components''
and then representing each medial component as obtained by attaching
surfaces with boundaries to $4$--valent graphs.  The two stages are
described by a two level extended graph structure.  The top level
describes a simplified form of the attaching of the irreducible medial
components to each other, and the second level extended graph
structure for each irreducible component specifies how to construct
the component.

We further use the data associated to the extended graph structures to
express topological invariants of $\Omega$ such as the homology and
fundamental group in terms of the singular invariants of $M$ defined
using the local standard types and the extended graph structures.
Using the classification, we characterize contractible regions in
terms of the extended graph structures and the associated data.
\end{abstract}

\maketitle

\section*{Introduction}  \label{S:sec0} 
We consider a compact region $\Omega \subset \R^3$ with generic smooth boundary $\cB$.  We are 
interested in geometric properties of $\Omega$ and $\cB$ which lie midway between their topology 
and geometry and can be summarized by the term \lq\lq geometric complexity\rq\rq.  Our goal 
in this paper is first to give a structure theorem for the \lq\lq geometric complexity\rq\rq\  
of such regions.  Second, we directly relate this structure to the topology of the region
and deduce how the topology places restrictions on the geometric complexity and how the 
structure capturing the geometric complexity determines the topology of the region.  

To explain what we mean by geometric complexity, we first consider $\R^2$.  If $\cB$ is a 
simple closed curve as in \fullref{fig.1.3}\,(a), then by the Jordan Curve and Schoenflies 
Theorems, $\Omega$ in \fullref{fig.1.3}\,(b) is topologically a $2$--disk, and hence 
contractible.  Hence, the region in \fullref{fig.1.3}\,(b) is topologically simple; 
however, it has considerable \lq\lq geometric complexity\rq\rq.  It is this geometric 
complexity which is important for understanding shape features for both $2$ and $3$ 
dimensional objects in a number of areas such as computer and medical imaging, biology, etc.  

\begin{figure}[ht!]
\labellist
\small\hair 2pt
\pinlabel {$\rm (a)$} [tr] at 104 433
\pinlabel {$\rm (b)$} [tr] at 289 433
\endlabellist
\centerline{\includegraphics[width=11cm]{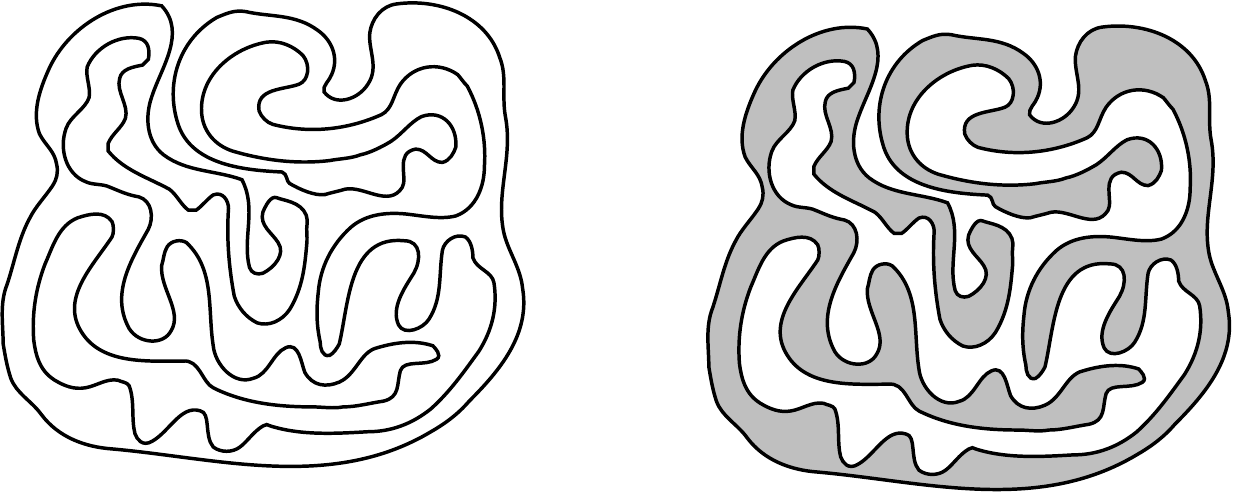}}
\caption{\label{fig.1.3}  Simple closed curve in (a) bounding a contractible  
region in $\R^2$ in (b)}
\end{figure}

This complexity is not captured by traditional geometrical invariants such as 
the local curvature of $\cB$ nor by global geometric invariants given by 
integrals over $\cB$ or $\Omega$.  Nor do traditional results such as the 
Riemann mapping theorem, which provides the existence of a conformal 
diffeomorphism of $D^2$ with $\Omega$, provide a comparison of the geometric 
complexity of $\Omega$ with $D^2$.  For $2$--dimensional contractible regions, 
a theorem of Grayson \cite{Gr}, building on the combined work of Gage and 
Hamilton \cite{Ga,GH}, shows that under curvature flow the boundary 
$\cB$ evolves so the fingers of the region shrink and the region ultimately 
simplifies and shrinks to a convex region which contracts to a \lq\lq round 
point\rq\rq.  Throughout this evolution, the boundary of the region remains 
smooth.  The order of shrinking and disappearing of subregions provides a 
model of the geometric complexity of the region.  Unfortunately, this 
approach already fails in $\R^3$, where the corresponding mean curvature 
flow may develop singularities, as in the case of the \lq\lq dumbell\rq\rq\, 
surface found by Grayson (also see eg Sethian \cite{Sn}).  

    The structure theorem for the \lq\lq geometric complexity\rq\rq\  
which we give is based on the global structure of the \lq\lq Blum medial 
axis\rq\rq\ $M$.  The medial axis is a singular space which encodes both the 
topology and geometry of the region (see \fullref{fig.1}).  It is defined in all 
dimensions and has multiple descriptions including locus of centers of 
spheres in $\Omega$ which are tangent to $\cB$ at two or more points (or have 
a degenerate tangency) as in Blum and Nagel \cite{BN}, the shock set for the eikonal/ \lq\lq 
grassfire\rq\rq\ flow as in Kimia, Tannenbaum and Zucker \cite{KTZ}, which is also a geometric flow on $\Omega$, 
and the Maxwell set for the family of distance to the boundary functions 
as in Mather \cite{M}.  It has alternately been called the central set by Yomdin \cite{Y}, and has as 
an analogue the cut-locus for regions without conjugate points in Riemannian 
manifolds.  

The multiple descriptions allow its local structure to be explicitly determined 
for regions in $\R^{n+1}$ with generic smooth boundaries: $M$ is an 
$n$--dimensional Whitney stratified set (by Mather \cite{M}) which is a strong 
deformation retract of $\Omega$.  The local structure of $M$ is given by a 
specific list of local models, by Blum and Nagel for $n = 1$, by Yomdin 
\cite{Y} for $n \leq 3$ (where it is called the central set) and Mather 
\cite{M} for $n \leq 6$, and is given a precise singularity theoretic 
geometrical description by Giblin \cite{Gb} for $n = 2$.  Results of Buchner 
\cite{Bu} show the cut locus has analogous properties for regions without 
conjugate points (and our structure theorem extends to this more general 
case).  Furthermore, by results in \cite{D1,D2,D3,D4}, 
we can derive the local, relative, and global geometry of both $\Omega$ and 
$\cB$ in terms of geometric properties defined on $M$.
 
\begin{figure}[ht!]
\labellist
\small\hair 2pt
\pinlabel $M$ at 289 646
\pinlabel $\Omega$ at 302 590 
\pinlabel $\cB$ [tl] at 468 627 
\pinlabel {$\rm cavity$} [l] at 407 682 
\endlabellist
\centerline{\includegraphics[width=9cm]{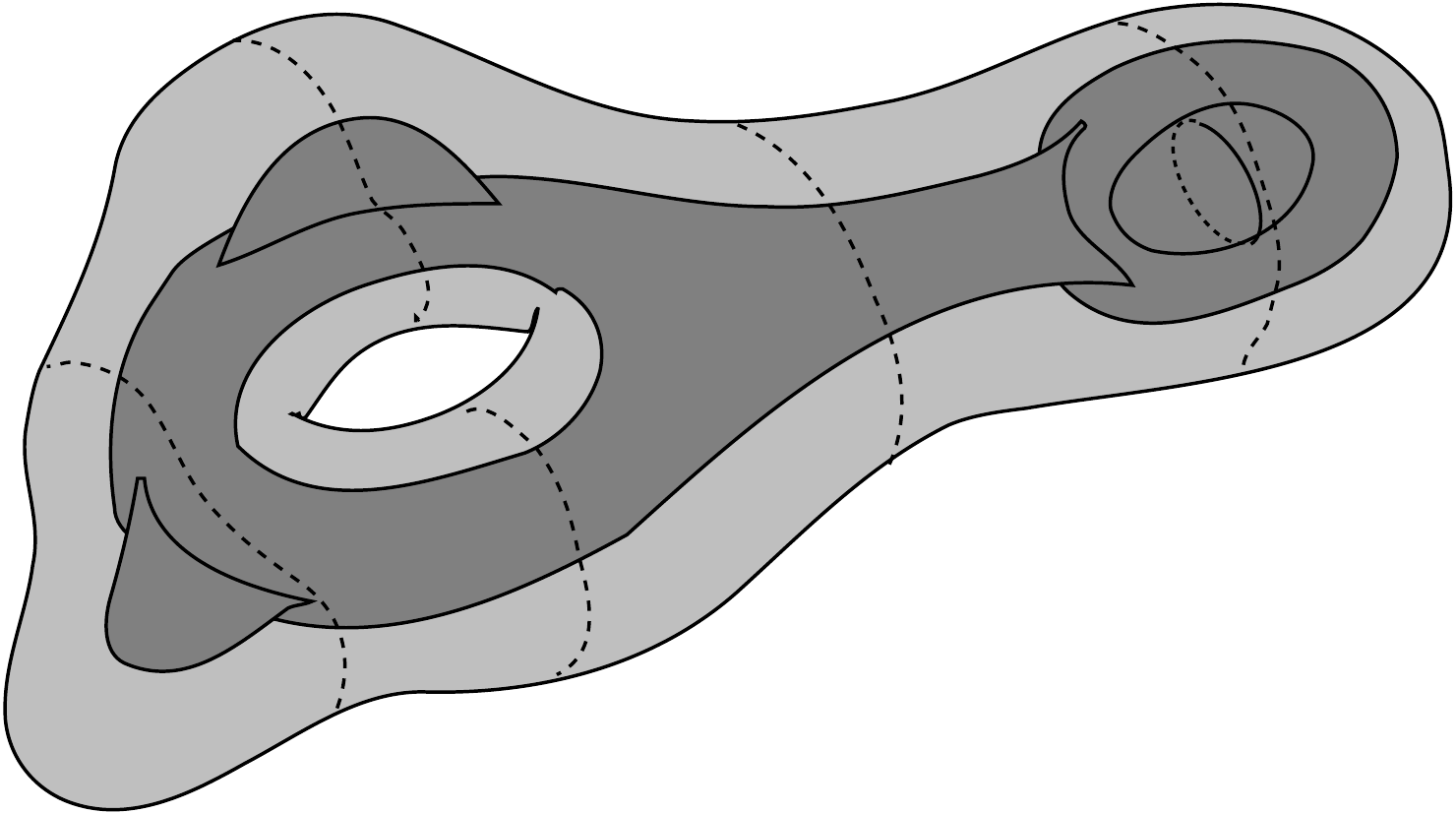}}
\caption{\label{fig.1}  Blum medial axis for a region in $\R^3$}
\end{figure}
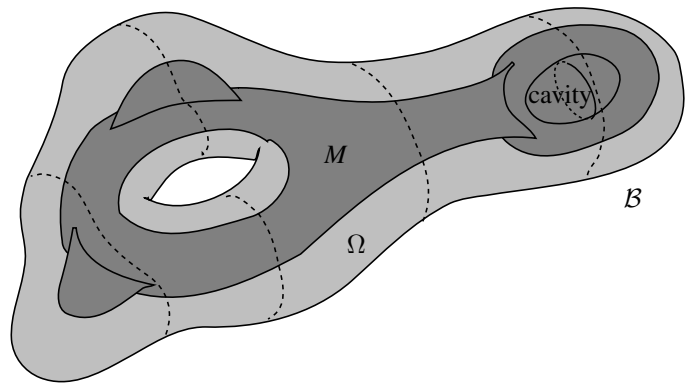

For example, for generic compact regions in $\R^2$, $M$ is a 
$1$--dimensional singular space whose singular points either have 
$Y\!$--shaped branching or are end points.  This defines a natural graph 
structure with vertices for the branch and end points and edges 
representing the curve segments joining these points (see \fullref{fig.1a}).  
This graph structure encodes the geometric complexity of the region.  
Furthermore, $\Omega$ is contractible if and only if the graph is a tree.  Then 
the tree structure can be used to contract the region (also for computer 
imaging a tree structure is a desirable feature, as trees can be searched in 
polynomial time).  A computer scientist Mads Nielsen has asked whether in 
the case of contractible $\Omega \subset \R^3$, the Blum medial axis still has 
a \lq\lq tree structure\rq\rq.  We answer this question as part of the 
structure theorem for the medial axis $M$ for generic regions.  

\begin{figure}[ht!]
\labellist
\small\hair 2pt
\pinlabel $M$ at 213 660
\pinlabel $\Omega$ at 194 688 
\pinlabel $\cB$ [bl] at 282 718
\endlabellist
\centerline{\includegraphics[width=10cm]{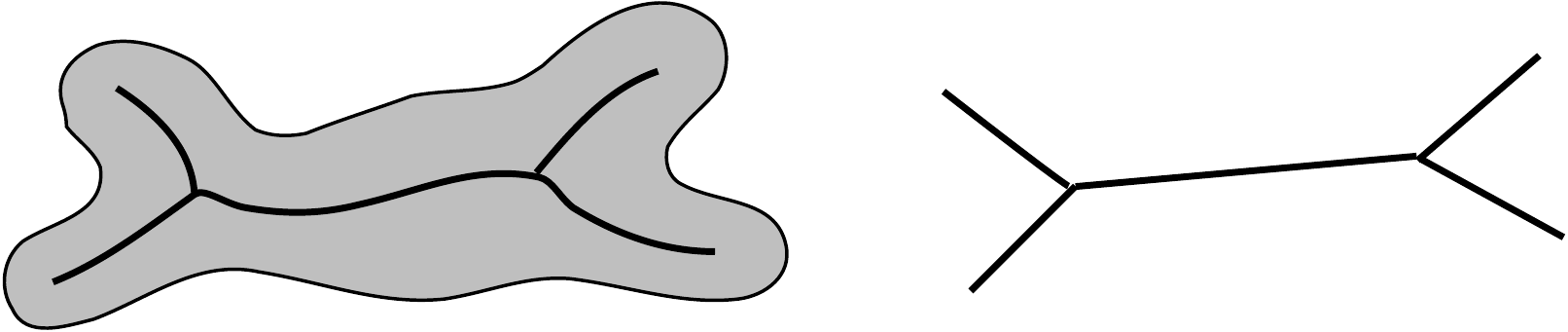}}
\caption{\label{fig.1a}  Blum medial axis for a region in $\R^2$ with 
associated graph structure}
\end{figure}

While the global structure of the medial axis as a Whitney stratified set with 
specific local singular structure does capture the geometric complexity of 
$\Omega$, it is insufficient in this raw form to characterize the geometric 
complexity or to see directly its relation with the topology.  This is what the 
structure theorem accomplishes.  

We give a brief overview of the form of the structure theorem (\fullref{Thm4.1}).  

\subsection*{The structure theorem for general regions} At the top level we decompose $M$ 
into \lq\lq irreducible medial components\rq\rq\ $M_i$ which are joined to each other along  
\lq\lq fin curves\rq\rq.  Geometrically this corresponds to taking connected sums of 
$3$--manifolds with boundaries.  A simplified form of the attaching is described by a {\it 
top level directed graph\/} $\Gamma(M)$.  We assign vertices representing the $M_i$, and 
directed edges from $M_i$ to $M_j$ for each edge component of $M_i$ attached to $M_j$ along a 
fin curve.  This is an \lq\lq extended graph\rq\rq \, in that there may be more than one edge 
between a pair of vertices or edges from a vertex to itself (see \fullref{fig.2}).  The 
resulting space obtained by attaching the $M_i$ as indicated by the edges of $\Gamma(M)$ is 
the simplified form $\hat M$ of the original $M$.  It is homotopy equivalent to $M$ and has 
the same irreducible components.  From the topological perspective, the decomposition is 
optimal in that the homology and fundamental group of $M$ are decomposed into direct sums, 
resp.\ free products, of those for the $M_i$.  Also, $M$ can be recovered from $\hat M$ by 
certain \lq\lq sliding operations along fin curves\rq\rq\ according to additional data defined 
from $M$.  

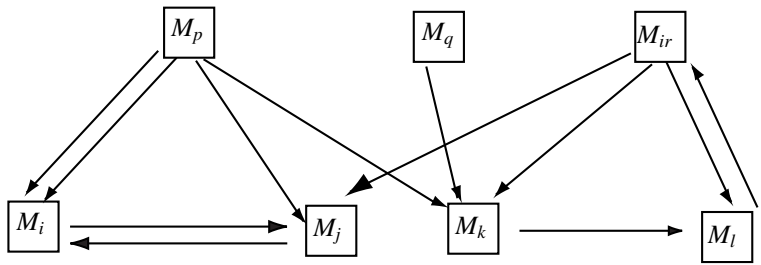
\begin{figure}[ht!]
\labellist
\small\hair 4pt
\pinlabel $M_p$ [tl] at 195 722 
\pinlabel $M_q$ [tl] at 315 720
\pinlabel {$M_{ir}$} [tl] at 418 719
\pinlabel $M_i$ [tl] at 121 629
\pinlabel $M_j$ [tl] at 263 626
\pinlabel $M_k$ [tl] at 331 627
\pinlabel $M_l$ [tl] at 452 623
\endlabellist
\centerline{\includegraphics[width=10cm]{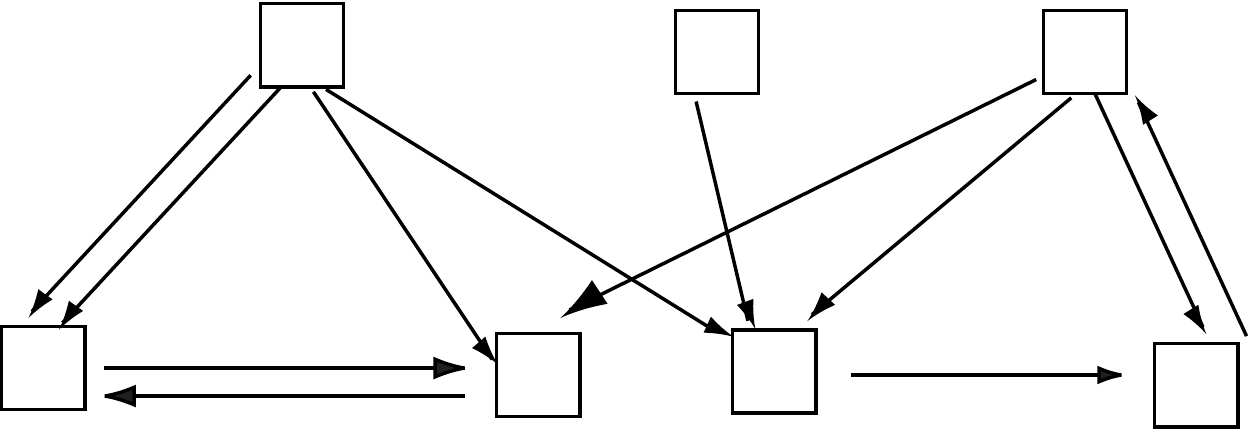}}
\caption{\label{fig.2} Graph structure given by decomposition into irreducible 
medial components with edges indicating the attaching along \lq\lq fin 
curves\rq\rq}
\end{figure}

At the second level, we describe the structure of each irreducible medial component $M_i$ by 
an extended graph $\gL(M_i)$, denoted more simply by $\gL_i$, as in \fullref{fig.2a}.  Here, 
there is an analogy with the resolution graph of an isolated surface singularity, whose 
vertices correspond to complex curves, ie compact orientable real surfaces, (with data the 
genera and self-intersection numbers) and edges indicating transverse intersection.  
Alternately, we can view the vertices as denoting multiply-punctured real surfaces which are 
attached to the set of intersection points as indicated by the edges.  For the irreducible 
medial components there is an analogous structure but with several added complications.  
 
The extended graph $\gL_i$ corresponding to an irreducible medial component $M_i$ has two 
types of vertices: \lq\lq $S$--vertices\rq\rq\, which are attached by edges to \lq\lq 
$Y\!$--nodes\rq\rq.  The $S$--vertices correspond to the medial sheets $S_{i j}$, which are 
(closures of) connected components of the set of smooth points of the component $M_i$ (which 
are not in general the connected components of the set of smooth points of $M$).  They are 
compact surfaces with boundaries (possibly nonorientable).  The $Y\!$--nodes correspond to 
connected components $\cY_{i j}$ of the \lq\lq $Y\!$--network\rq\rq\ $\cY_i$ of $M_i$.  The 
$Y\!$--network of $M_i$ is the collection of \lq\lq $Y\!$--branch curves\rq\rq\ together with 
vertices which are the \lq\lq $6$--junction points\rq\rq\ where $6$ sheets of the Blum medial 
axis come together in a point along $4$ $Y\!$--branch curves (and is described by a $4$--valent 
extended graph).  Each edge from a vertex for $S_{i j}$ to a node for $\cY_{i k}$ represents 
the attaching of the sheet $S_{i j}$ along one of its boundary components to the component 
$\cY_{i k}$.  

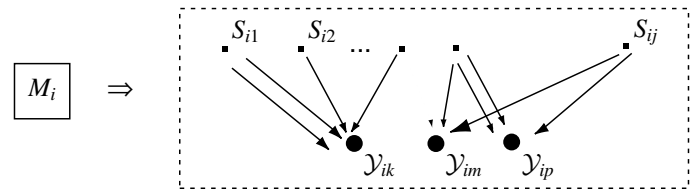
\begin{figure}[ht!]
\labellist
\small\hair 2pt
\pinlabel $M_i$ [t] at 113 685
\pinlabel $S_{i1}$ [bl] at 237 704
\pinlabel $S_{i2}$ [bl] at 288 704
\pinlabel $S_{ij}$ [bl] at 506 702
\pinlabel $\cY_{ik}$ [tl] at 325 634 
\pinlabel $\cY_{im}$ [tl] at 380 634
\pinlabel $\cY_{ip}$ [tl] at 431 634
\endlabellist
\centerline{\includegraphics[width=9cm]{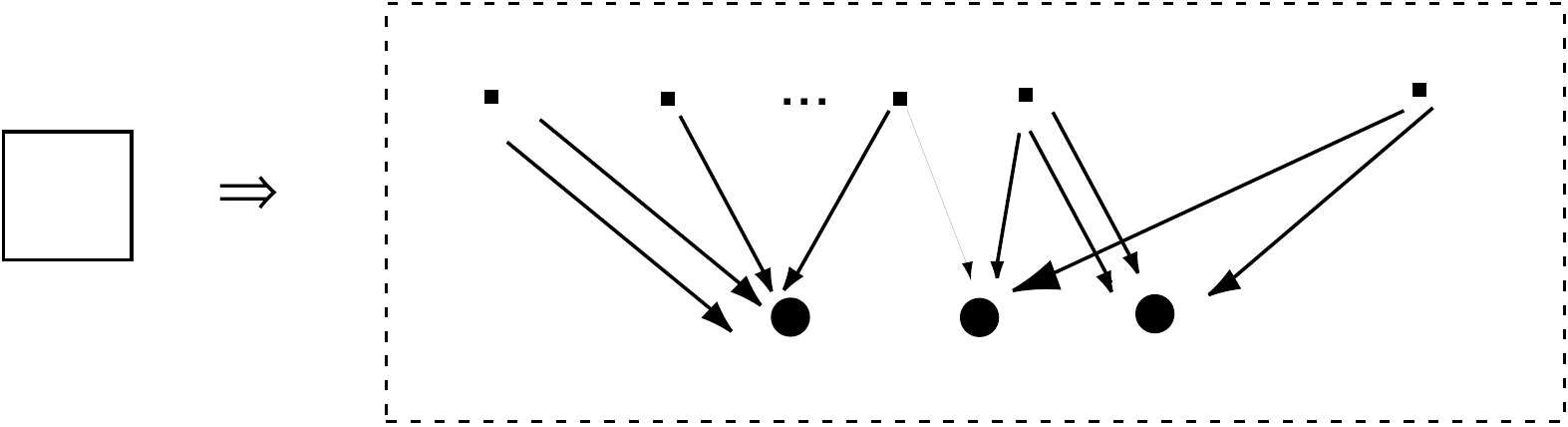}}
\caption{\label{fig.2a} Extended graph structure for irreducible medial 
component $M_i$: $S$--vertices $\bullet$ representing connected medial 
sheets $S_{i j}$ joined to $Y\!$--nodes $\scriptstyle\blacksquare$ representing 
components $\cY_{i k}$ of the \lq\lq  $Y\!$--network\rq\rq}
\end{figure}

There is associated data attached to the vertices and edges of $\gL_i$; 
namely, for $S$--vertices the genus, orientability, and number of medial edge 
curves on $S_{i j}$, for $Y\!$--nodes the $4$--valent extended graph associated 
to $\cY_{i k}$, and for each edge the attaching data of the boundary 
component of $S_{i j}$ to $\cY_{i k}$.  

From this structure we obtain two alternate ways to view $M_i$: either as 
obtained by attaching compact surfaces with boundaries along (some of the) 
boundary components or by an associated CW--decomposition.  These allow 
us to relate the graph structures and their associated data for all of the 
$M_i$ to the topological structure of $\Omega$.  

Topological invariants of $M$, and hence $\Omega$, are computed by \fullref{Thm3.3}  and \fullref{Thm3.5} in terms of $M_i$ and 
$\Gamma(M)$.  In turn, those of each $M_i$ are expressed in terms of the 
extended graph $\gL_i$, and the associated data.  This data includes the  
singular invariants of each $M_i$ such as the numbers of $6$--junction 
points, medial edge curves, components of $\cY_i$, the genera of $S_{i j}$, 
etc.  From this we define a single \lq\lq algebraic attaching map\rq\rq\ for 
$M_i$.  From this data, we compute in \fullref{Thm3.3} the topological 
invariants of $M_i$ such as homology groups and Euler characteristic.  We 
then deduce in \fullref{Thm3.5} the corresponding invariants for $M$ 
and hence, $\Omega$.  Also, in \fullref{Thm7.5} we give a presentation of 
$\pi_1(M_i)$ in terms of the associated data for $\gL_i$.  These provide 
topological bounds on the geometric complexity of $\Omega$.   

The answer to Nielsen\rq s question turns out in general to be no provided we wish to 
classify the medial axis up to homeomorphism.  The tree has to be replaced by an \fullref{DecmpAlg} we 
give in \fullref{S:sec1} for decomposing the medial axis into the irreducible components 
$M_i$.  The data of the algorithm requires more than just a graph showing which medial 
components are attached to which.  However, the simplified version $\hat M$, which sacrifices 
some of the detail of the attachings, is described by an extended graph $\gG(M)$, and for 
contractible regions $\gG(M)$ is a tree.  Then the special form which the structure theorem 
takes for contractible regions in $\R^3$ (\fullref{Thm5.1}) gives a complete 
characterization of contractible regions by the following conditions: the extended graphs 
$\gG(M)$ and the $\gL_i$ are trees, the medial sheets have genus $0$ (and so are topological 
$2$--disks with a finite number of holes) with at most one boundary curve representing a 
medial edge curve, a numerical \lq\lq Euler relation\rq\rq\ is satisfied which involves the 
basic medial invariants, and a fundamental group condition holds (\fullref{Thm5.0}) which 
involves a space formed from $\cY_i$ and $\gL_i$.  

The author is especially grateful to Mads Nielsen for initially raising the question about 
the tree structure for the contractible case, which led this investigator to these questions 
and results.  

\subsubsection*{Acknowledgements}  This research was partially supported by grants from the 
National Science Foundation DMS-0405947 and CCR-0310546 and a grant from DARPA.

\section{Decomposition into irreducible medial components}  
\label{S:sec1}
\subsection*{Generic local structure of Blum medial axis}
    We consider a region $\Omega \subset \R^3$ with generic smooth 
boundary $\cB$ and Blum medial axis $M$.  Then by Mather \cite{M}, $M$ can 
be viewed as the Maxwell set for the family of distance functions on $\cB$; 
hence, for generic $\cB$, it is a $2$--dimensional Whitney stratified set.  
Also, the generic local structure has one of the following local forms in \fullref{fig.4} (see eg \cite{Gb}, where Giblin gives a very explicit 
geometric description).  
  
\begin{figure}[ht!]
\labellist
\small\hair 2pt
\pinlabel {{\rm (a)\qua edge}} [t] at 71 638
\pinlabel {{\rm (b)\qua Y--branching}} [t] at 189 638
\pinlabel {{\rm (c)\qua fin creation point}} [t] at 332 638 
\pinlabel {{\rm (d)\qua ``$6$--junction''}} [t] at 474 638
\endlabellist
\centerline{\includegraphics[width=12cm]{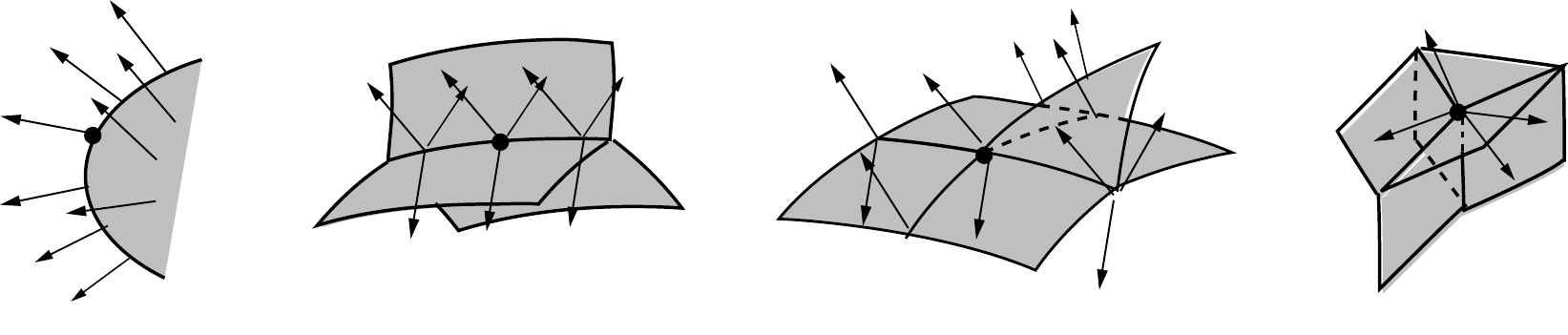}}
\caption{\label{fig.4} Local generic structure for Blum medial axes in $\R^3$ 
and the associated radial vector fields to points of tangency on the 
boundary}
\end{figure}
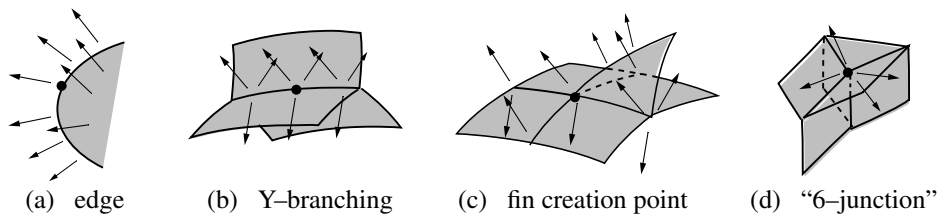 

Then $M$ consists of the following: (i) smooth connected ($2$--dimensional) strata; 
$1$--dimensional strata consisting of  (ii) $Y\!$--junction curves along which three strata meet 
in a $Y\!$--branching pattern and (iii) edge curves consisting of edge points of $M$; 
$0$--dimensional strata consisting of (iv) fin points and (v) $6$--junction points, where six 
medial sheets meet along with 4 $Y\!$--junction curves.  Connected components of $Y\!$--junction 
curves end either at fin points or $6$--junction points; while edge curves only end at fin 
points.  We refer to the union of $Y\!$--junction curves, fin points, and $6$--junction points 
as the \textit{initial $Y\!$--network} $\cY$.  

 In the simplest form we can view the medial 
axis as formed by attaching the connected $2$--dimensional smooth strata to the $1$--complex 
formed from the $Y\!$--branch and medial edge curves and the fin and $6$--junction points.  In 
fact, this approach misses a considerable amount of global structure for subspaces of $M$.  
To identify this larger structure we will decompose $M$ into \lq\lq irreducible medial 
components\rq\rq\ $M_i$ by \lq\lq cutting $M$ along fin curves\rq\rq.  Then we further 
decompose the irreducible medial components $M_i$ by representing them as obtained by 
attaching smooth medial sheets to the resulting $Y\!$--network formed from the union of 
$Y\!$--junction curves and $6$--junction points in $M_i$.  From this we will ultimately give a 
CW--decomposition to compute the topological invariants of $M$.  

\subsection*{Cutting the medial axis along fin curves}  
In order to proceed, we first explain how we cut along fin curves.  

 Suppose we have on 
the Blum medial axis a fin point $x_0$.  Then near $x_0$ we can distinguish the \lq\lq fin 
sheet\rq\rq\ which is the sheet that contains a medial axis edge curve ending at $x_0$.  We 
can begin following the $Y\!$--branch curve from $x_0$, while keeping track of the fin sheet 
which remains a connected sheet as we move along the curve.  Eventually one of two things 
must happen: either we reach a $6$--junction point or another fin point.  

 First, if the 
$Y\!$--branch curve meets a $6$--junction point, then because the fin sheet is locally connected 
near the $6$--junction point, we can follow the edge of the sheet as it continues through the 
$6$--junction point.  After the $6$--junction point we have identified both the corresponding 
continuation of the $Y\!$--branch curve, and the sheet.  We can do this for each $6$--junction 
point it encounters.  As $M$ is compact, eventually the $Y\!$--branch curve must meet another 
fin point.  We shall refer to the edge of the sheet from one fin point to the other as a 
\textit{fin curve}.  

 At the end of the fin curve, what was identified as the fin sheet 
(close to the $Y\!$--branch curve) from the beginning may or may not be the fin sheet for the 
end point.  If the sheet is a fin sheet at both ends, then we refer to the fin curve as being 
\lq\lq essential\rq\rq, while if it is only a fin sheet at one end, then we refer to the fin 
curve as \lq\lq inessential\rq\rq, (later discussion will explain the reason for these 
labels).  Examples of these are shown in \fullref{fig.7}.  

\begin{figure}[ht!]
\labellist
\small\hair 2pt
\pinlabel {(a)} [tr] at 36 655 
\pinlabel {(b)} [tr] at 307 655
\endlabellist
\centerline{\includegraphics[width=12cm]{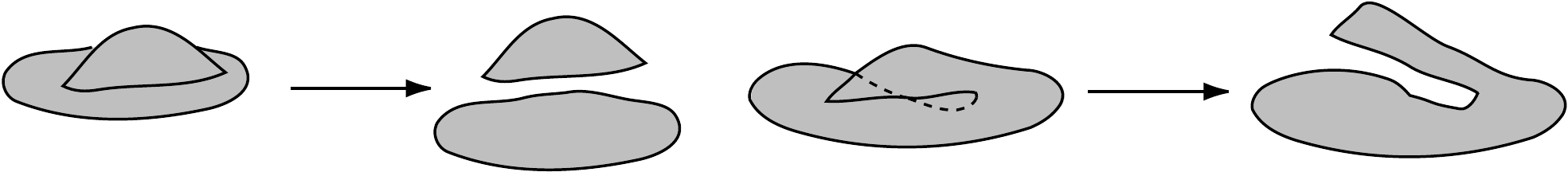}}
\caption{\label{fig.7} Two possibilities for fin curves on a medial sheet and 
the results of cutting along the fin sheets: (a) essential fin curve   (b)  
inessential fin curve}
\end{figure} 

An example of a region containing an inessential fin curve is given in \fullref{fig.8} and might be called a \lq\lq Mobius board\rq\rq, a surf board but 
with a \lq\lq Mobius band\rq\rq\  twist. 

\begin{figure}[ht!]
\centerline{\includegraphics[width=6cm]{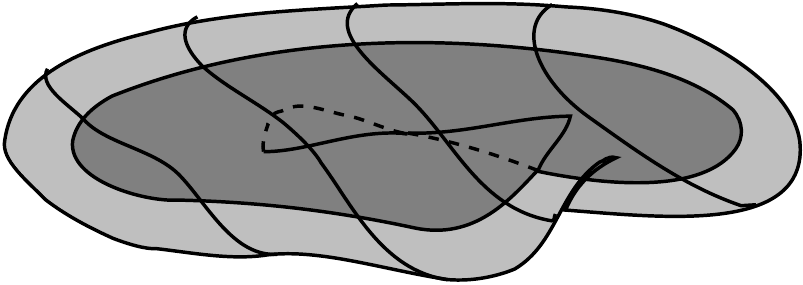}}
\caption{\label{fig.8} \lq\lq Mobius board\rq\rq\ with a \lq\lq inessential fin 
curve\rq\rq}
\end{figure} 

Then to cut along a fin curve beginning from a fin point $x_0$, we locally 
disconnect the fin sheet from the direction of $x_0$ by locally adding a point 
of closure to the fin sheet for each point of closure on the fin curve (as 
followed from $x_0$).  We must take care as a point on the fin curve can 
locally be a point of closure for more than one part of the sheet.  After this 
step, locally at an added closure point, the fin sheet is now a surface with 
piecewise smooth boundary; see \fullref{fig.7}.  If the fin curve is essential, 
then the fin sheet is now locally disconnected from the remaining two sheets 
still attached along the $Y\!$--junction curve.  While if the fin curve is 
inessential, the fin sheet is still attached at the other fin point to the 
remaining two sheets; again see \fullref{fig.7}\,(b).  

Then we can take the two remaining sheets still attached along that fin 
curve and smooth them to form a smooth sheet along the curve, with 
former $6$--junction points on the fin curve becoming $Y\!$--branch points 
(for another $Y\!$--branch curve).  

After having cut along the fin curves as described, the former points on the 
fin curve have become altered as follows:  
\begin{enumerate}
\item Fin points become points of an edge of the fin sheet.
\item  A former $Y\!$--branch point becomes a closure point on an edge of a fin 
sheet.
\item  A former $6$--junction point become a (topological) $Y\!$--network 
point.
\item  A $Y\!$--branch point on a base sheet becomes a (topological) 
$2$--manifold point of that sheet.  
\end{enumerate}  
Because of (4) above, the end result depends upon a further distinction for essential fin 
curves.  A  \textit{type--$1$ essential fin curve} will be one which only intersects other 
essential fin curves at $6$--junction points; otherwise, it shares a segment of $Y\!$--branch 
curve with another essential fin curve, and it will be \textit{type--$2$ essential fin curve} 
(see eg \fullref{fig.7b}).  If we cut along a type--$1$ essential fin curve, then the fin 
sheet becomes disconnected from the other sheets (at least along the curve) and this does not 
alter any other essential fin curve.  If we cut along a type--$2$ essential fin curve, then it 
will alter the structure of the other essential fin curves sharing a segment of $Y\!$--branch 
curve with it.  Hence, we can prescribe the following algorithm for decomposing the medial 
axis.  
\begin{TitleExample}\setobjecttype{TitEx}\label{DecmpAlg}~
\begin{enumerate}
\item Identify all type--$1$ essential fin curves and systematically cut along 
type--$1$ essential fin curves (it does not matter which order we choose as 
cutting along one does not alter the fin properties of another).  
\item After cutting along all type--$1$ essential fin curves, we may change 
certain inessential fin curves to type--$1$ essential ones.  If so return to step (1).   
\item There only remain type--$2$ essential fin curves and inessential fin 
curves.  Choose an essential fin curve and cut along it.  If a type--$1$ 
essential fin curve is created, return to step (1).  Otherwise, repeat this step 
until no essential fin curves remain.
\item  When there are no other essential fin curves, choose an inessential 
fin curve which crosses a $6$--junction point, and cut it from one side until 
we cut across one $6$--junction point.  
\item  Check whether we have created an essential fin curve.  If so then we 
cut along it, and repeat the earlier steps (1)--(3).  
\item If no essential fin curve is created, then we repeat step (3) 
until there are only inessential fin curves which do not cross 
$6$--junction points.  
\item  Finally we can cut each such remaining inessential fin curve, producing 
part of a smooth sheet as in \fullref{fig.7}\,(b).  
\item The remaining connected pieces are the \lq\lq irreducible medial 
components\rq\rq\ $M_i$ of $M$.  
\end{enumerate}
\end{TitleExample} 

\begin{Remark}
\label{Rem2.1a}
The distinct connected pieces created following steps (1) and (2) are intrinsic 
to $M$; while those created using steps (3) and (4) are not because choices 
are involved.  Which choices are made typically depends on the given 
situation and the importance we subjectively assign to how sheets are 
attached.  
\end{Remark}  
 
\begin{Example}
\label{Ex2.1b}
In \fullref{fig.7b}\,(a), we have a contractible medial axis with a pair of 
type--$2$ essential fin curves.  Depending on which essential fin curve we 
choose, $1$-$2$ or $3$-$4$, we choose to cut along in step (3), we obtain 
either (b) or (c), which leads to different attachings (and hence top level 
graph) for the irreducible medial components.  An alternate possibility would 
be to cut each fin sheet along the fin curves and view them as being 
attached partially along the edge of a fourth sheet. Again, the exact 
geometric form of $M$ may suggest one choice being preferred over the 
others. 
\end{Example}

\begin{figure}[ht!]
\labellist
\small\hair5pt
\pinlabel {(a)} [tr] at 67 577
\hair 2pt
\pinlabel $1$ [l] at 114 582 
\pinlabel $2$ [l] at 196 610
\pinlabel $3$ [r] at 86 589
\pinlabel $4$ [r] at 167 621
\pinlabel {(b)} <-3pt,-2pt> at 287 654
\pinlabel $1$ [tr] at 453 615
\pinlabel $2$ [l] at 543 645
\pinlabel $3$ [r] at 343 673
\pinlabel $4$ [tl] at 410 711
\pinlabel {(c)} <-3pt,7pt>  at 287 531 
\pinlabel $1$ [br] at 425 493 
\pinlabel $2$ [tl] at 507 524
\pinlabel $3$ [r] at 308 530 
\pinlabel $4$ [r] at 390 560 
\endlabellist
\centerline{\includegraphics[width=10cm]{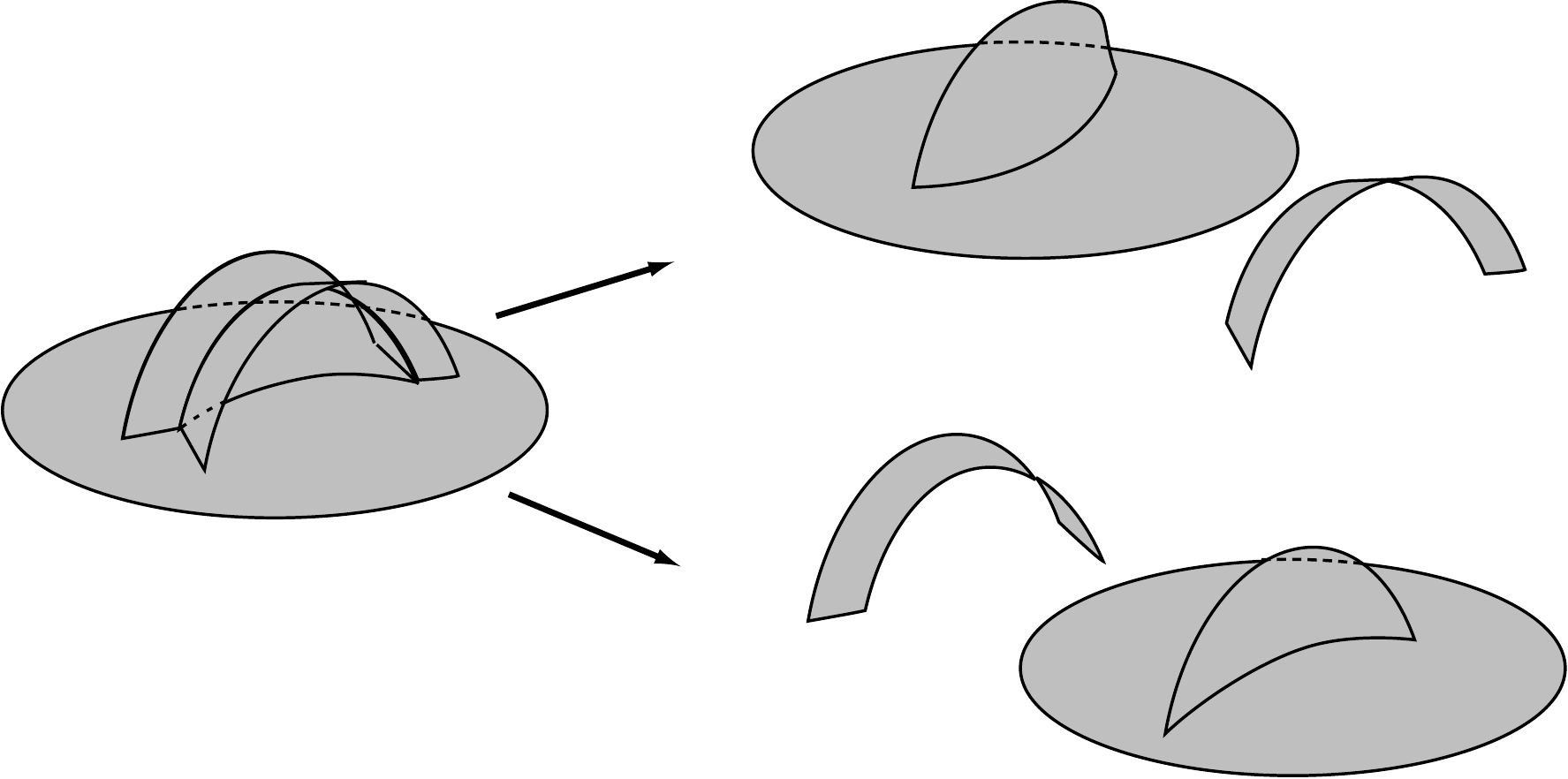}}
\caption{\label{fig.7b} Nonuniqueness of medial decomposition 
resulting from type--$2$ essential fin curves: (a) is a contractible medial axis 
with only type--$2$ essential fin curves, and (b) and (c) illustrate the results 
from cutting along the fin curves $1$-$2$ or $3$-$4$.}  
\end{figure}
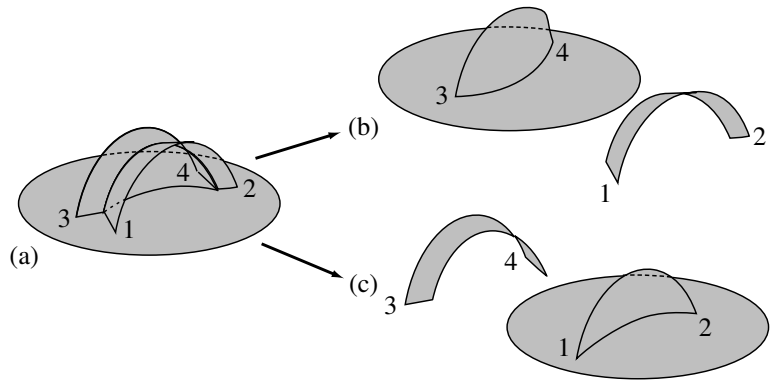 

\begin{Example}
\label{Ex2.1a}
In \fullref{fig.7a}\,(a), we have a contractible medial axis with $10$ fin 
points $1$-$10$, and all fin curves are inessential.  Depending on how we 
choose cuts in step (4) of the algorithm, we can end up with $1$, $2$, or $3$ 
irreducible medial components.  

If we cut from $4$ through the first $6$--junction point, then $3$-$6$ 
becomes an essential fin curve, and we cut away the fin sheet $M_1$ as in  
\fullref{fig.7a}\,(b).   Then further cutting from $8$ through the first 
$6$--junction point, we create another essential fin curve $2$-$9$.  
Cutting along 
it creates a second fin sheet $M_2$. The remaining inessential fin curves 
$1$-$4$, $5$-$7$, and $8$-$10$ can be contracted to points on edges of the 
third sheet $M_3$.  Each of these $3$ medial sheets are then irreducible 
components.  

Alternatively, after the first cut, we could have instead cut from $9$, and 
then from $4$ again, and then only inessential fin curves remain without  
$6$--junction points, so they contract to a second sheet, and we only obtain 
two irreducible components.  Thirdly, we could have begun cutting from $7$, 
then $8$, and then $4$ twice  and we would obtain only a single medial sheet 
with inessential fin curves, leading to a single irreducible component.  
\end{Example}

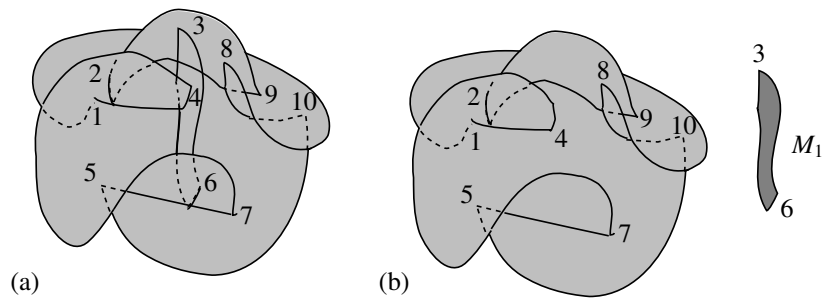
\begin{figure}[ht!]
\labellist
\small\hair 2pt
\pinlabel {(a)} [tr] at 49 397 
\pinlabel $1$ [t] at 82 499
\pinlabel $2$ [r] at 88 517
\pinlabel $3$  at 144 546
\pinlabel $4$ [tr] at 148 510
\pinlabel $5$ [br] at 84 448
\pinlabel $6$ [l] at 144 450
\pinlabel $7$ [l] at 166 430
\pinlabel $8$ [b] at 161 523
\pinlabel $9$ [l] at 181 504
\pinlabel $10$ [b] at 208 491
\pinlabel {(b)} [tr] at 275 397
\pinlabel $1$ [t] at 313 484
\pinlabel $2$ [r] at 319 503
\pinlabel $3$ [b] at 486 520
\pinlabel $4$ [tl] at 359 485
\pinlabel $5$ [br] at 314 435
\pinlabel $6$ [l] at 496 434
\pinlabel $7$ [l] at 397 419
\pinlabel $8$ [b] at 391 510
\pinlabel $9$ [l] at 410 489
\pinlabel $10$ [b] at 440 476
\pinlabel $M_1$ [l] at 503 472
\endlabellist
\centerline{\includegraphics[width=10cm]{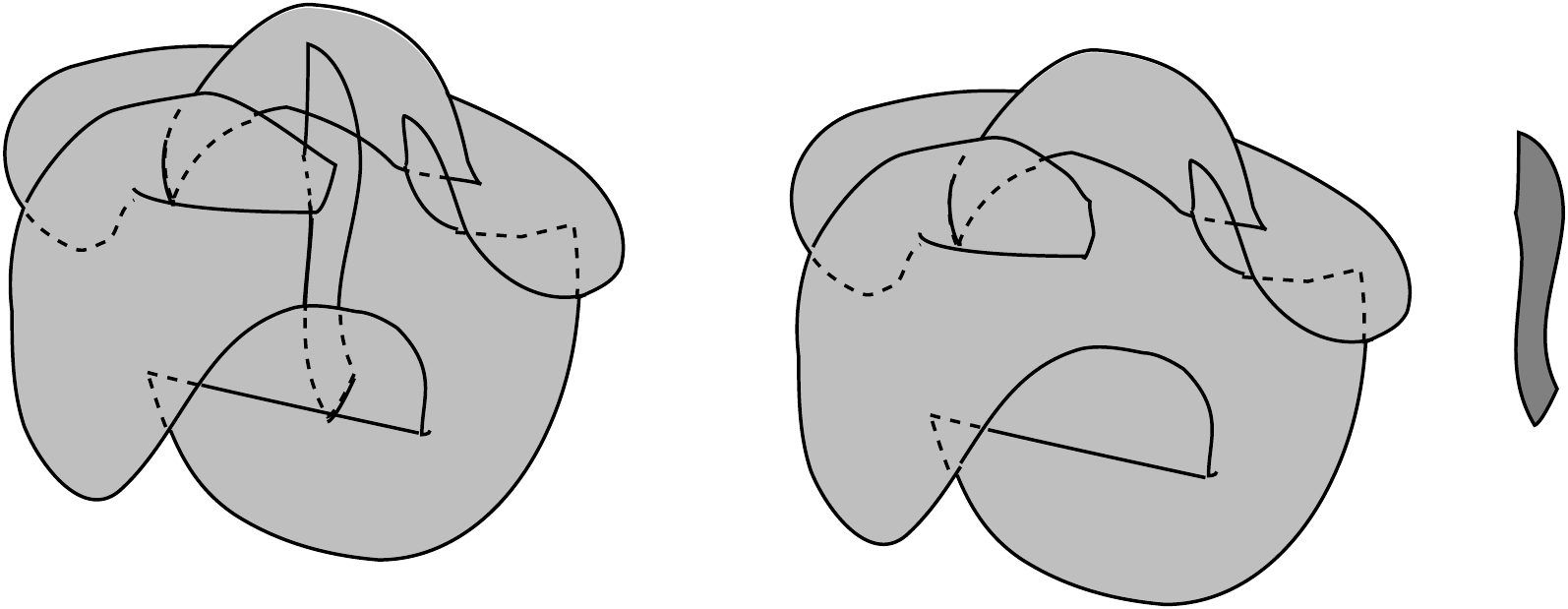}}
\caption{\label{fig.7a} Nonuniqueness of medial decomposition 
resulting from inessential fin curves: (a) is a contractible medial axis with 
only inessential fin curves, and (b) illustrates the cutting of irreducible medial 
component $M_1$ after first cutting from fin point $4$.}
\end{figure} 

\subsection*{Constructing the medial axis by attaching irreducible medial 
components along fin curves} 

To reverse the algorithm and reconstruct $M$ from the $M_i$ requires the 
following additional attaching data:  
\begin{enumerate}
\item the list of which segments of medial edge curves which will be 
attached as either essential or inessential fin curves
\item  the order in which the sheets will be attached  
\item  the curve at the $j$--th stage to which the attaching will be made as a 
fin curve to obtain the $(j+1)$--st stage.
\end{enumerate}
The curves in step (3) may cross multiple components, and need only be well- defined up to 
isotopy; however, this isotopy is for the space obtained at the $j$--th stage, whose singular 
set must be preserved by the isotopy.  Hence for example, in \fullref{fig.8c}, even though 
the medial axis is contractible, if the attaching curve were on the same side of the two 
holes then the resulting spaces would not be homeomorphic.  

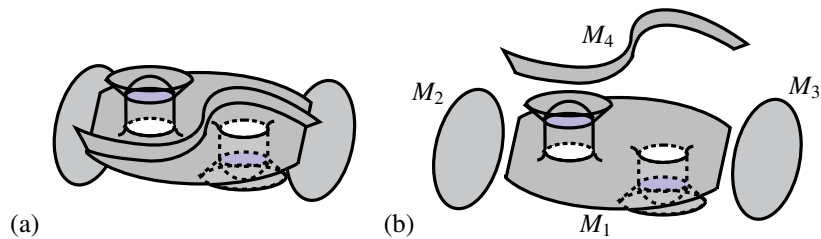
\begin{figure}[ht!]
\labellist
\small\hair 2pt
\pinlabel {(a)} [tr] at 63 617
\pinlabel {(b)} [tr] at 297 617 
\pinlabel $M_1$ [t] at 408 617 
\pinlabel $M_2$ [br] at 316 681
\pinlabel $M_3$ [bl] at 526 685
\pinlabel $M_4$ [br] at 422 717
\endlabellist
\centerline{\includegraphics[width=10cm]{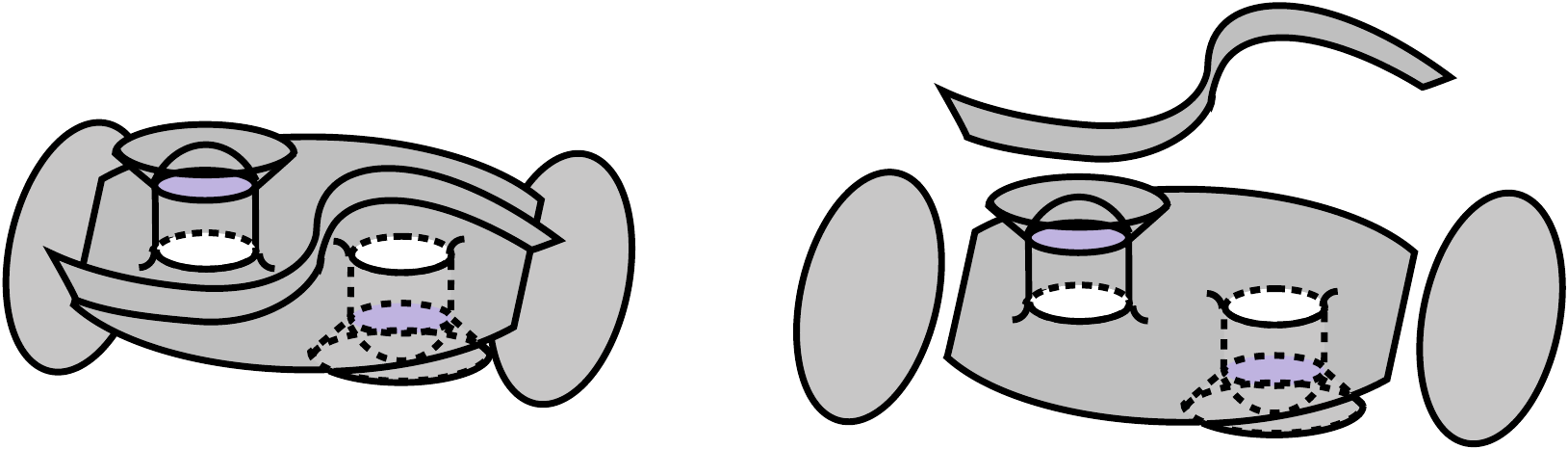}}
\caption{\label{fig.8c} (a) A contractible medial axis, with attaching of  
component $M_4$ to multiple components $M_1$, $M_2$, $M_3$\quad (b) 
Irreducible medial components for (a)}
\end{figure}  

\section{Extended graph structure and medial decomposition}  
\label{S:sec2} We next turn to a simplified version of $M$ using the same irreducible medial 
components, but with simplified attaching.  The resulting simplified version of $M$, which we 
denote by $\hat M$, will still be homotopy equivalent to $M$.  To define $\hat M$, we 
consider an alternate way to understand the topological effects of attaching along a fin 
curve $\gamma$.  Instead, we isotope the attaching map along the support of the original 
$\gamma$, by sliding and shrinking so it now becomes a fin curve $\gamma^{\prime}$  which no 
longer passes through any $6$--junction points (which have disappeared after the isotopy as 
in \fullref{fig.6}).  

\begin{figure}[ht!]
\centerline{\includegraphics[width=10cm]{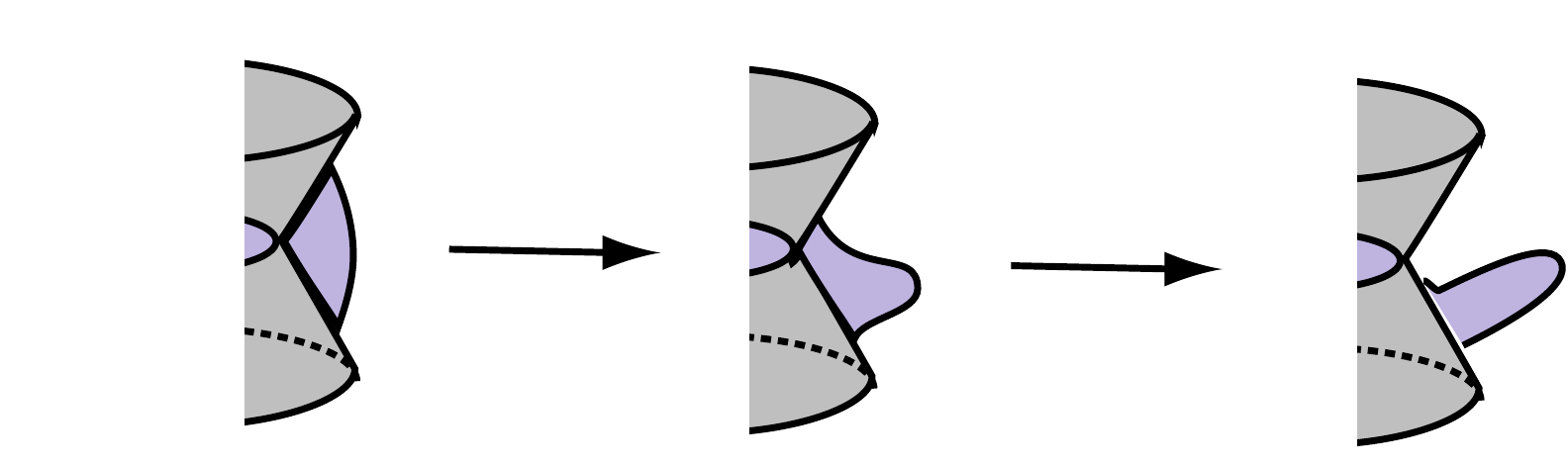}}
\caption{\label{fig.6} Isotoping a fin curve so it misses $6$--junction points}
\end{figure} 

We can achieve this by replacing $M$ by a subcomplex which is a strong 
deformation retract, and only differs from $M$ in a small neighborhood of 
the fin curve.  
\begin{Lemma}[Isotopy lemma for fin curves]
\label{Lem2.1} Let $\gamma$ be a fin curve in $M$.  Then there is a Whitney stratified set 
$M^{\prime} \subset M$, which is a strong deformation retract of $M$ and which only differs 
from $M$ in a given neighborhood of $\gamma$.  Furthermore, in that neighborhood, $\gamma$ 
has been replaced by a fin curve $\gamma^{\prime}$ which does not meet a $6$--junction point.  
\end{Lemma}
\begin{proof}
We obtain $M^{\prime}$ as the result of a series of deformation retractions 
$M \supset M^{(1)} \supset M^{(2)} \supset  \dots \supset M^{(k)} =  
M^{\prime}$.   Here $M^{(j)} \supset M^{(j+1)}$ corresponds to either 
moving the fin point along a $Y\!$--branch curve so it is in a neighborhood of a 
$6$--junction point where it has the normal form as in \fullref{fig.7c}\,(a), 
or when $M^{(j)}$ already has this form then $M^{(j+1)}$ is a deformation 
retraction across a $6$--junction point as in \fullref{fig.7c}\,(b).  

    In the first case we may construct the deformation because $M$ is 
analytically trivial along a $Y\!$--branch curve and so analytically a product in a 
neighborhood of a compact segment of a $Y\!$--branch curve.  In the second 
case, we may use the normal form for $6$--junction points to deform along 
the shown region. 
\end{proof}

\begin{figure}[ht!]
\labellist
\small\hair 2pt
\pinlabel {(a)} [tr] at 40 653 
\pinlabel {(b)} [tr] at 327 653
\endlabellist
\centerline{\quad \includegraphics[width=12cm]{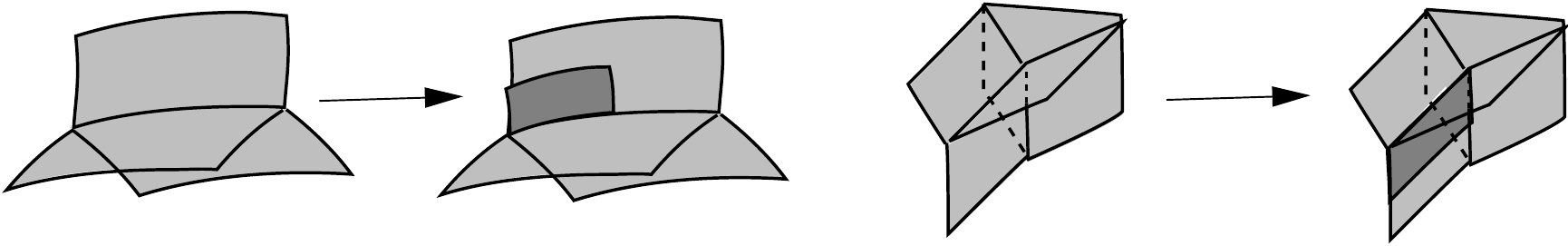}}
\caption{\label{fig.7c} (a)  Deformation retraction along a $Y\!$--branch curve.
\quad 
(b) Deformation retraction across a $6$--junction point.  In both cases 
the deformation is across the dark region.}
\end{figure} 

We view this process as sliding and shrinking the fin curve along its support, 
across $6$--junction points until the fin curve no longer crosses 
$6$--junction points.  Also, after isotoping the fin curves, the resulting spaces are 
still homotopy equivalent.  

After first isotoping the fin curves, we can more clearly see the topological 
effect of cutting along fin curves.  
\begin{Lemma}
\label{Lem2.2} In the preceding situation, let $\gamma^{\prime}$ be a fin curve which does 
not meet a $6$--junction point.  There are two possibilities.  

{\rm(1)}\qua   Suppose the fin curve $\gamma^{\prime}$ is essential, with fin sheet 
$S_i$ at both ends.  Then when we cut $S_i$ along $\gamma^{\prime}$, the 
sheet locally becomes disconnected from the remaining sheets of the fin 
curve as in \fullref{fig.7}\,(a).  

{\rm(2)}\qua  If instead the fin curve $\gamma^{\prime}$ is inessential, then we cut all 
three sheets along the fin curve; we obtain a single sheet with boundary edge 
containing the fin curve as in \fullref{fig.7}\,(b).  This may be alternately 
obtained by shrinking the fin curve to a point.  \end{Lemma}
\begin{proof}
In the essential case, by the triviality of $M$ along $Y\!$--branch sheets, the fin 
sheet remains a single sheet along the fin curve.  When we cut it we 
disconnect it from the base sheet as in \fullref{fig.7}\,(a).  By contrast, in 
the inessential case when we cut the fin sheet, the edge curve continues 
along the cut fin curve, joining up with another edge curve again.  Hence, the 
fin points become edge points for a single edge curve.  

As an alternate way to view the process, the fin sheet for one fin point 
becomes part of the base sheet.  Then we may follow a path around the 
second fin point to end up on the base sheet for the first fin point.  Then we 
may continue the curve around the first fin point to be on the sheet which 
becomes the fin sheet for the second fin point.  If we extend these curves 
so they intersect the edge curve, then together they form a closed curve, 
so that when the fin curve is contracted to a point we obtain a 
$2$--disk, one edge of which is the edge curve.  
\end{proof}

\begin{Remark}
\label{Rem2.3}
This Lemma contrasts cutting along essential versus inessential fin curves.  
When we cut along the essential fin curve we disconnect the fin sheet from 
the base sheet, causing a change in homotopy type as in \fullref{fig.7}\,(a).  
By contrast, for inessential fin curves, when we contract the fin curve to a 
point $p$, the sheet then has an edge curve which passes through $p$.  In 
this case, the \lq\lq fin curve\rq\rq\ can be eliminated, and the pair of fin 
points cancelled, without any change in the homotopy type of $M$.  
\end{Remark}

Now beginning with a Blum medial axis, we can repeat the  
\fullref{DecmpAlg} given in \fullref{S:sec1}, except at each step, instead of 
cutting along the fin curve, by the Isotopy Lemma for Fin Curves, we slide 
the fin sheet along the fin curve.  We pass all of the $6$--junction points 
(which then become $Y\!$--branch points), until the fin curve lies on a single 
smooth sheet.  Then by \fullref{Lem2.2}\,(b), we may contract the 
inessential fin curves eliminating their fin points.  The resulting space with 
these simplified attachings is $\hat M$.  Finally we cut the essential fin 
curves using \fullref{Lem2.2}\,(b).  The resulting connected 
components $M_i$ are again the irreducible medial components of $M$.  
 
The initial $Y\!$--network has been altered by the removal of the fin curves to 
yield the $Y\!$--network $\cY$.  Then $\cY = \smash{\tsty\bigcup_i \cY_i}$, where each 
$\cY_i$ is the resulting $Y\!$--network for $M_i$.  Then $\hat M$ is obtained 
from the irreducible medial components by attaching their edges which came 
from essential fin curves to the isotoped positions in some $M_j$.  There is 
some ambiguity in this construction in addition to that coming from the 
algorithm.  We make a choice of the medial sheet that the fin curve passes 
through.  However, the irreducible medial components and the homotopy type 
of $\hat M$ will remain the same.  

This gives rise to a \lq\lq top level directed extended graph\rq\rq\, 
$\Gamma(M)$ whose vertices correspond to the $M_i$, with an edge going 
from $M_i$ to $M_j$ for each segment of an edge curve of $M_i$ attached 
to $M_j$ along a fin curve.  

There is the following relation between the topology of the full medial axis 
$M$ and that of the irreducible medial components $\{ M_i\}_{i = 1}^r$ 
and the top level extended graph $\Gamma(M)$.
\begin{Proposition}
\label{Prop2.4}
In the preceding situation
\begin{enumerate}
\item   for any coefficient group $G$, 
$$   H_j(M; G) \qua \simeq  \qua  H_j(\Gamma(M); G) \,\,  \oplus \,  
(\oplus_{i =1}^r \, H_j(M_i; G)  \quad \mbox{ for all } j > 0;   $$
\item  $   \pi_1(M) \qua  \simeq  \qua  \pi_1(\Gamma(M)) \,\, * \, (*_{i 
=1}^r  \, \pi_1(M_i))$.
\end{enumerate}
\end{Proposition}

Of course the graph $\Gamma(M)$ only contributes to homology in dimension 
$1$.  Although it is an extended graph, we can compute its fundamental 
group just as for graphs.  To prove this we also need the next proposition.
\begin{Proposition}
\label{Prop2.5}
If $\Gamma$ is a nonempty connected extended graph, then it contains a 
maximal tree $T$.  Then $\pi_1(\Gamma)$ is a free group with one 
generator for each edge (including loops) of $\Gamma$ which do not belong 
to $T$.  
\end{Proposition}
\begin{proof}[Proof of \fullref{Prop2.5}]
First, we construct a  connected graph $\Gamma^{\prime}$ from 
$\Gamma$.  For each pair of vertices of $\Gamma$ joined by an edge we 
remove all but one of the edges; as well we remove any loops (edges from a 
vertex to itself).  What remains is a connected graph $\Gamma^{\prime}$.  
Then such a graph has a maximal connected tree $T$ (by eg Spanier \cite[Chapter 
2]{Sp}).  This is also a maximal connected tree for $\Gamma$.  

There is only one special case which occurs for $Y\!$--networks, and hence which we must allow. 
It is a single $S^1$ without any vertex.  Then we have to artificially introduce a vertex so 
$S^1$ becomes a loop on the vertex.  Then the maximal tree is just the introduced vertex.  
However, we emphasize that we will not count such \lq\lq artificial vertices\rq\rq \, for 
later numerical relations involving the number of vertices.

Then by repeating the same proof in \cite[Chapter 2]{Sp} as for the 
fundamental groups of graphs, $\pi_1(\Gamma)$ is a free group with one 
generator for each edge (including loops) of $\Gamma$ which do not belong 
to $T$.
\end{proof}
We obtain as an immediate corollary. 
\begin{Corollary}
\label{Cor2.6}
If $M$ is contractible, then so is each $M_i$ contractible, and furthermore, 
$\Gamma(M)$ is a tree.  
\end{Corollary}
\begin{proof}[Proof of \fullref{Cor2.6}]
  If $M$ is contractible, then $M$ is connected, $\pi_1(M) = 0$, and $H_j(M) = 
0$ for all $j > 0$.  Thus,  by (2) of \fullref{Prop2.4},  
$\pi_1(\Gamma(M)) = 0$; thus, $\Gamma(M)$ is a tree.  In addition, for each 
$i$, by the combination of (1) and (2) of \fullref{Prop2.4}, 
$\pi_1(M_i) = 0$, and $H_j(M_i) = 0$ for all $j > 0$.  Also, by definition, each 
$M_i$ is connected.  Then by the Hurewicz Theorem, $\pi_j(M_i) = 0$ for all 
$j \geq 0$.  However, $M_i$ is a CW--complex, so by another Theorem of 
Hurewicz, $M_i$ is contractible.  
\end{proof}
\begin{proof}[Proof of \fullref{Prop2.4}]
By our earlier discussion, $M$ is homotopy equivalent to the space $\hat M$ 
obtained by attaching $M_i$ to $M_j$ along an edge segment of $M_i$ to a 
smooth sheet of $M_j$.  However, this edge segment can be homotoped to a 
point so $M$ is instead homotopy equivalent to the space obtained by 
attaching an external curve segment  $\ga_{i j}$ from $x_{i j}$ on the edge 
of $M_i$ to a point $y_{i j}$ in a smooth sheet of $M_j$.  Then from one 
fixed point $x_{i 0}$ on $M_i$ we can choose disjoint curves $\gb_{i j}$ from 
$x_{i 0}$ to the other points $x_{i j}$.  Then we choose a maximal tree $T$ 
in $\Gamma(M)$ and choose the attaching curves $\ga_{i j}$  corresponding 
to the edges of $T$, as well as the curves $\gb_{i j}$ in each $M_i$, which 
give a tree from each $x_{i 0}$.  The resulting space $X$ which is the union 
of these trees and the $M_i$ is homotopy equivalent to the pointed union of 
the $M_i$.  Hence, 
$$  \pi_1(X) \,\, \simeq \,\,  *_{i =1}^r  \, \pi_1(M_i)  \qquad \mbox { and } 
\qquad  H_j(X) \,\, \simeq \,\,  \oplus_{i =1}^r  \, H_j(M_i) \,\, \mbox{ for } 
j > 0.   $$
Then as in the computation of $\pi_1$ of a graph as in \cite[Chapter 
2]{Sp}, we inductively add one of the remaining edges and show using 
Seifert--Van Kampen that we are taking a free product with a free group on 
one generator.  After $k$ steps, where $k$ is the number of additional edges 
(including loops) of $\Gamma(M)$ not in $T$, we obtain that $\pi_1(M)$ is 
isomorphic to the free product of $\pi_1(X)$ with a free group on $k$ 
generators, which is exactly the fundamental group of $\Gamma(M)$.  

For the case of homology we instead use Mayer--Vietoris in a similar argument.
\end{proof}
\fullref{Prop2.4} is valid for any region $\Omega$.  Thus, we have a 
relation between the topology of  $M$, and hence $\Omega$ (as $M$ is a strong 
deformation retract of $\Omega$), and the individual irreducible medial 
components $M_i$ and the top level extended graph $\Gamma(M)$.  In the 
special case that $\Omega$ is contractible, then so also is $M$ and each $M_i$, 
and $\Gamma(M)$ is a directed tree. 

\begin{Example}[Knot complement regions]
\label{Ex2.7}
The simplest example of contractible $\Omega$ corresponds to boundary $\cB 
\simeq S^2$.  A slightly more complicated example would have $\cB$ 
consisting of two components $S^2$ and a torus $T^2$.  In the case that 
$T^2$ belongs to the compact region bounded by $S^2$ in $\R^3$, then 
$\Omega$ is the region inside $S^2$ and outside of the torus $T^2$.  This is a 
\lq\lq knot complement region\rq\rq.  

If $K \subset \R^3$ is a knot, then the complement $C = \R^3 \backslash 
K$ is not a region in our sense.  However, we can take a small tube 
$T_\gevar$ around the knot and a large sphere $S^2_R$ containing the tube.  
Then the region $\Omega = D^2_R \backslash \mbox{int}(T_\gevar)$ is a 
compact region with boundary $S^2_R \cup \partial T_\gevar$.  Also, $\Omega 
\subset C$ is a strong deformation retract.  For a generic knot with $\gevar 
> 0$ and $R$ sufficiently large, the Blum medial axis of $\Omega$ will be generic.  

To see what form the irreducible medial components $\{ M_i\}$ take for 
such a knot complement region, we note a few results from knot theory.  
First, by Alexander duality, $H_i(\Omega) \simeq \Z$ for $i < 3$ and $0$ 
otherwise.  Also, $\pi_1(C)$ does not split as a free product of nontrivial 
groups.  Hence, by \fullref{Prop2.4} there are just three 
possibilities: (i) $\pi_1(\gG(M)) \neq 0$ so all $M_i$ are simply connected; or 
$\gG(M)$ is a tree and all but one of the $M_i$, say $M_1$, are simply 
connected; and then either (ii) $H_2(M_1) \simeq \Z$, or (iii) $H_2(M_1) = 0$ 
and for a single other $M_i$, say $M_2$, $H_2(M_2) \simeq \Z$.  

In the first case, we must have $\pi_1(C) \simeq \pi_1(\gG(M)) \simeq \Z$, 
which implies $K$ is the unknot.  In this case, all but one of $M_i$, say 
$M_1$, have trivial reduced homology and hence are contractible, while 
$M_1$ is simply connected with $0$ reduced homology except for $H_2(M_1) 
\simeq \Z$.  Then $\gG(M)$ is homotopy equivalent to $S^1$, $M_1$ is 
homotopy equivalent to $S^2$, and $M$ (and hence $C$) is homotopy 
equivalent to the pointed union $S^1 \vee S^2$  (which is homotopy 
equivalent to a standard homotopy model for the complement of the unknot, 
the torus with a disk attached to one of the generators of $\pi_1$).  

In the second case, $\gG(M)$ is a tree and all of the other $M_i$ except 
$M_1$ have trivial fundamental group and reduced homology, and hence are 
contractible.  There is then exactly one noncontractible irreducible medial 
component $M_1$, and $C$ is homotopy equivalent to it.  

In the third case, all but $M_1$ and $M_2$ are contractible, $M_2$ is 
homotopy equivalent to $S^2$, and $\gG(M)$ is a tree.  Then $C$ is 
homotopy equivalent to $M_1 \vee M_2  \simeq  M_1 \vee S^2$.  

In all of these cases, all but one or two $M_i$ are contractible and so 
contribute nothing to the topology of the region but do contribute to its 
geometric complexity.  Next, we shall determine the structure of the 
$M_i$ and relate this to their topology.
\end{Example}

\section{Structure of irreducible medial components}  
\label{S:sec3}  Next, we decompose each irreducible medial component $M_i$, which is 
connected and locally has only $Y\!$--junction points, edge points, and $6$--junction points 
(but no fin points).  We may apply the local closure procedure used in \fullref{S:sec1} to cut 
the sheets of $M_i$ along the remaining $Y\!$--network as in \fullref{fig.5}.  We do this by 
locally adding closure points to each of the three sheets meeting at a point of the 
$Y\!$--network, or to each of the $6$ sheets meeting at $6$--junction points.  This disconnects 
the connected components of the set of smooth points of $M_i$ from the $Y\!$--network, turning 
them into compact connected surfaces with piecewise smooth boundaries.  We refer to each such 
compact surface with boundary $S$ as a \textit{medial sheet} of $M_i$.  The data attached to 
each medial sheet consists of the triple $(g, o, e)$.  Here $g$ is the genus of $S$ (after 
attaching disks to the boundary components), $o = 1$ or $0$ corresponding to whether $S$ is 
orientable or nonorientable (a nonorientable surface with at least one boundary component can 
be embedded in $\R^3$), and $e$ denotes the number of boundary components of $S$ (ie $S^1$\rq 
s) which are medial edge curves of $M$ (keeping in mind that part of the edge might have 
originally been an essential fin curve).  

 Then each $M_i$ is formed by attaching 
certain of the boundary components of medial sheets in $M_i$ to the $Y\!$--network $\cY_i$ of 
$M_i$, reversing the cutting process above.  The $Y\!$--network $\cY_i$ has connected 
components $\smash{\{ \cY_{i k}\}_{k =1}^{c_i}}$, and we denote the collection of medial sheets of 
$M_i$ by $\smash{\{ S_{i j}\}_{j =1}^{s_i}}$.

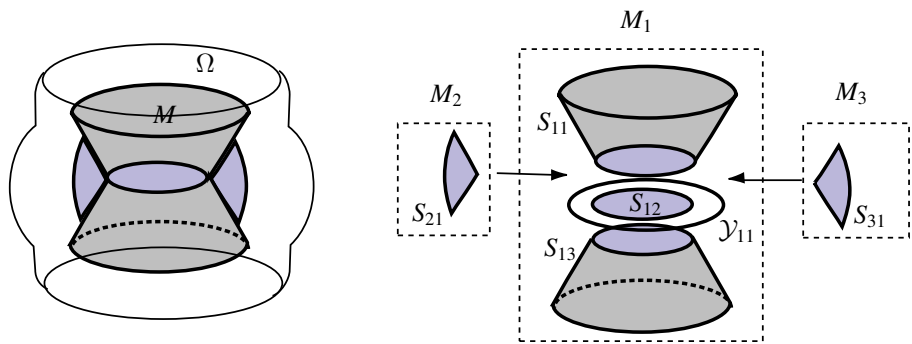
\begin{figure}[ht!]
\labellist
\small\hair 2pt
\pinlabel $M$ at 138 684
\pinlabel $\Omega$ at 161 713
\pinlabel $M_2$ at 298 695
\pinlabel {$S_{21}$} [r] at 300 627  
\pinlabel {$S_{11}$} [r] at 371 679
\pinlabel {$S_{13}$} [r] at 375 607
\pinlabel $M_1$ [b] at 406 728
\pinlabel {$S_{12}$} at 411 633
\pinlabel {$\cY_{11}$} [tl] at 450 626
\pinlabel $M_3$ at 528 696
\pinlabel {$S_{31}$} [l] at 526 626 
\endlabellist
\centerline{\includegraphics[width=12cm]{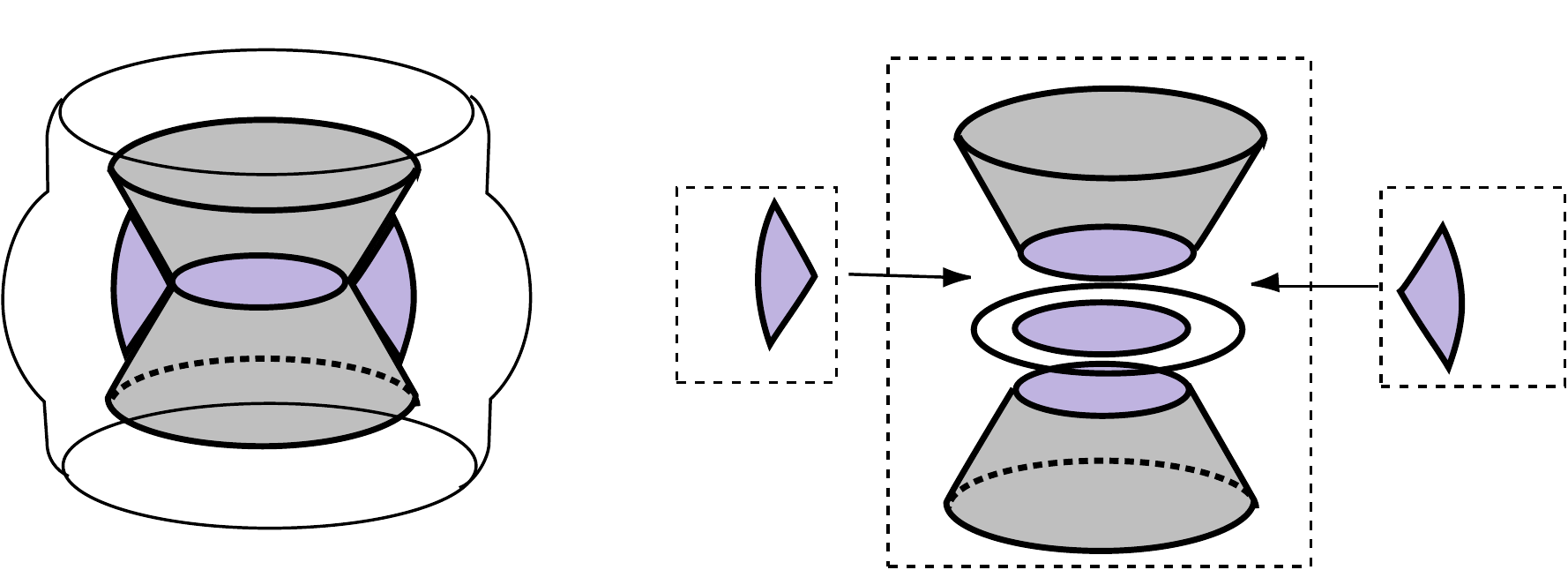}}
\caption{\label{fig.5} Cutting the medial axis first along fin curves and then 
along the $Y\!$--network $\cY$}
\end{figure} 

We assign an extended graph $\gL_i$ to each such irreducible medial component $M_i$ as 
follows.   The graph will have two type of vertices: to each sheet $S_{i j}$ will be 
associated an $S$--vertex, and to each connected component $\cY_{i k}$ we associate a vertex 
which we refer to as a $Y\!$--node.  We will associate an edge from the $S$--vertex associated 
to $S_{i j}$ to the $Y\!$--node corresponding to $\cY_{i k}$, for each boundary component of 
$S_{i j}$ (ie an $S^1$) which is attached to $\cY_{i k}$.   

 Each medial sheet $S_{i j}$ 
has the associated data $(g, o, e)$.  Second, each connected component $\cY_{i k}$ of the 
$Y\!$--network $\cY_i$ can be described by an extended graph $\gP_{i k}$.  Its vertices 
correspond to the $6$--junction points of $\cY_{i k}$, and there are distinct edges between 
vertices which correspond to the distinct $Y\!$--junction curves of $\cY$ joining the 
corresponding $6$--junction points.  Because a $6$--junction point has exactly four 
$Y\!$--junction curves ending at it, each vertex of the graph will have four edges.  Thus, each 
vertex will have a valence of $4$, and we will refer to such an extended graph as a \lq\lq 
$4$--valent extended graph\rq\rq.   

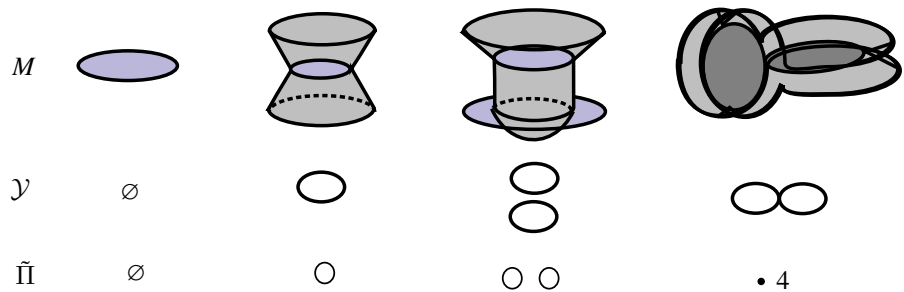
\begin{figure}[ht!]
\labellist
\small\hair 2pt
\pinlabel $M$ [r] at 72 700 
\pinlabel $\cY$ [r] at 69 630
\pinlabel {$\tilde \gP$} [r] at 72 582
\pinlabel $4$ [l] at 489 578
\endlabellist
\cl{\includegraphics[width=11cm]{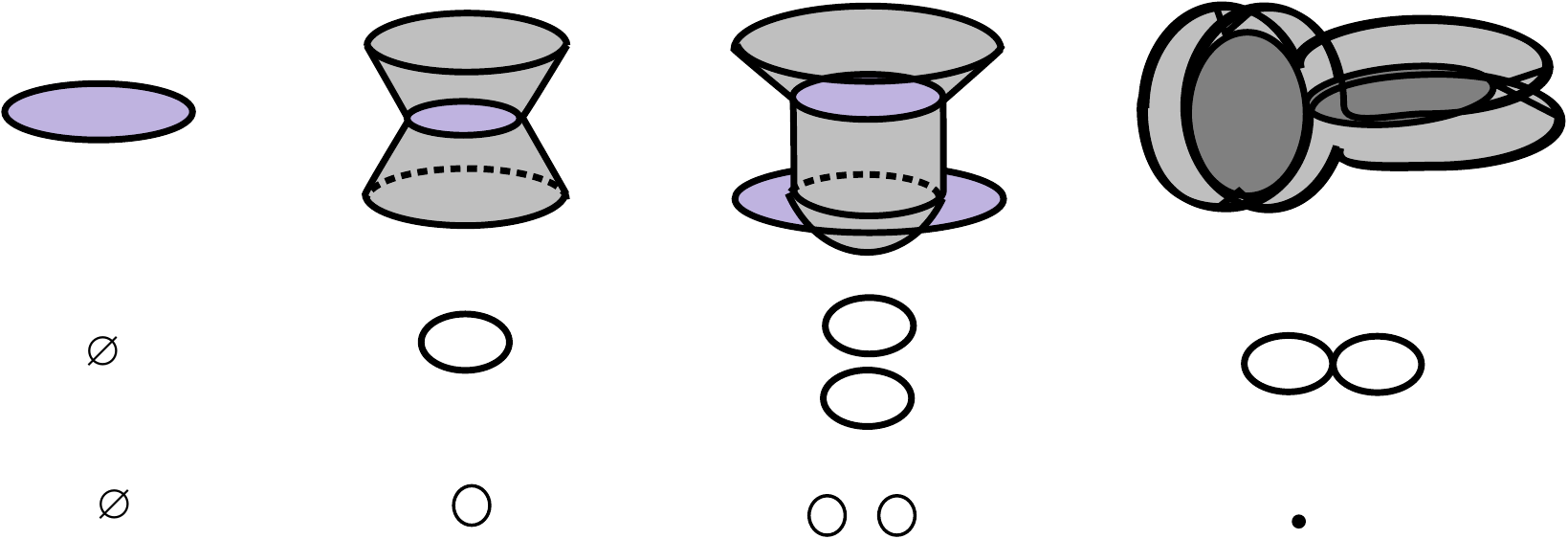}}
\caption{\label{fig.9} Simplest examples of irreducible medial components, 
their $Y\!$--networks and the reduced network graphs}
\end{figure} 

A basic result we need for $4$--valent extended graphs is a consequence of 
the following result for $m$--valent extended graphs.

\begin{Proposition}
\label{Prop3.1}
If $\gP$ is a nonempty connected $m$--valent extended graph with $k$ 
vertices, then ($m$ or $k$ must be even and) $\pi_1(\gP)$ is a free group 
on $\frac{k}{2}\cdot (m - 2) + 1$ generators.  
\end{Proposition}

Hence, for a nonempty connected $4$--valent graph $\gP$ with $k$ 
vertices,  $\pi_1(\gP)$ is a free group on $k + 1$ generators.
\begin{proof}
An extended graph $\gP$ is still a $1$--complex.  Hence, $\pi_1(\gP)$ is a 
free group on $\ell$ generators, where $\ell = \mathrm{rk} H_1(\gP)$.  If V, 
respectively E, denotes the number of vertices, respectively edges, of 
$\gP$, then by the $m$--valence of $\gP$, $2E = mV$.  As $V = k$, 
\begin{align*}
\ell \quad &= \quad 1 - \chi(\gP) \quad = \quad 1 - (V - E)   \\
&= \quad 1 - (k - \frac{k}{2}\cdot m) 
\end{align*}
which equals the value claimed.
\end{proof}  

\subsection*{Associated weighted graphs of extended graphs}
In order to work with an extended graph $\Gamma$, we simplify the description by defining an 
associated weighted graph $\tilde \Gamma$.   $\tilde \Gamma$ will have the same vertices as 
$\Gamma$; however, for each pair of vertices with at least one edge between them, we remove 
all but one of the edges.  Likewise, we remove all loops.  Then we assign to each vertex 
$v_i$  the number $2\ell$ where there are $\ell$ loops at $v_i$.  Also for the edge between 
vertices $v_i$ and $v_j$, we assign the integer $m$ which is the number of edges in $\Gamma$ 
between $v_i$ and $v_j$.  Then the valence at $v_i$ is $2\ell + \sum m_i$ where we sum $m_i$ 
over the edges of $\tilde \Gamma$ which end at $v_i$. If for a vertex the number of loops is 
$0$ then we suppress it; likewise if there is only $1$ edge between two vertices, we usually 
suppress the $1$.  For extended graph $\gP_{i k}$ we refer to this graph $\tilde \gP_{i k}$ 
as the \textit{reduced $Y\!$--network graph} of $\cY_{i k}$.  For such a reduced graph, at a 
vertex the sum of the vertex number and edge numbers always equals $4$.  There are very few 
possibilities for vertices as shown in \fullref{fig.11}

\begin{figure}[ht!]
\labellist
\small\hair 2pt
\pinlabel $4$ [l] at 89 620
\pinlabel $4$ [b] at 192 703 
\pinlabel $3$ [b] at 187 660
\pinlabel $2$ [b] at 193 615
\pinlabel $2$ [r] at 164 614
\pinlabel $2$ [b] at 335 705
\pinlabel $2$ [b] at 385 705
\pinlabel $2$ [b] at 359 662
\pinlabel $2$ [b] at 350 614
\endlabellist
\centerline{\includegraphics[width=12cm]{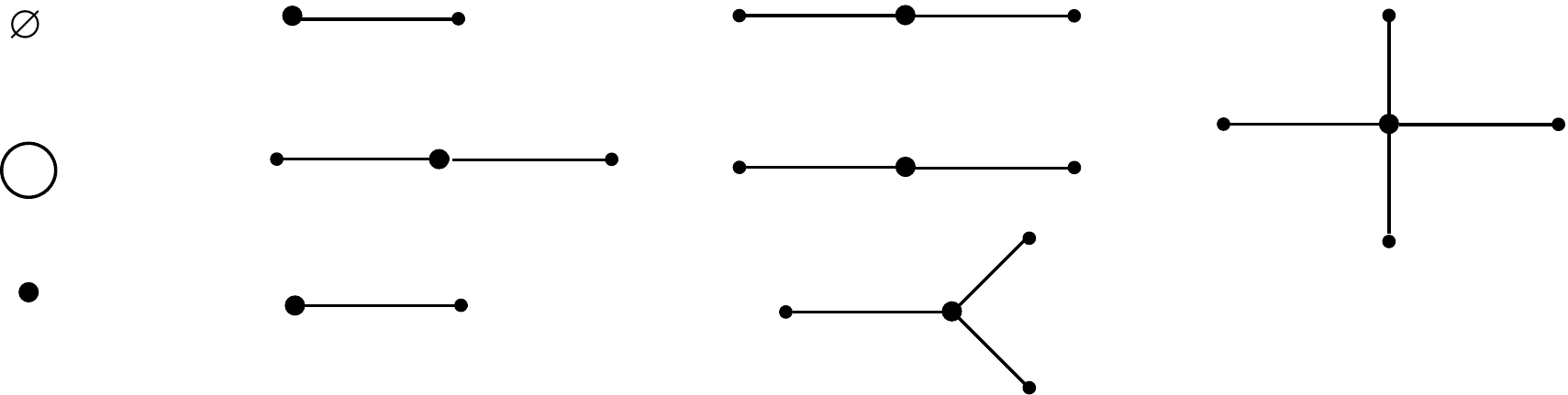}}
\caption{\label{fig.11} The possible local vertex structures (at the enlarged 
vertex) for reduced $Y\!$--network graphs.  The case of $S^1$ is the 
exceptional case of nonempty $Y\!$--network without any vertices.}
\end{figure}
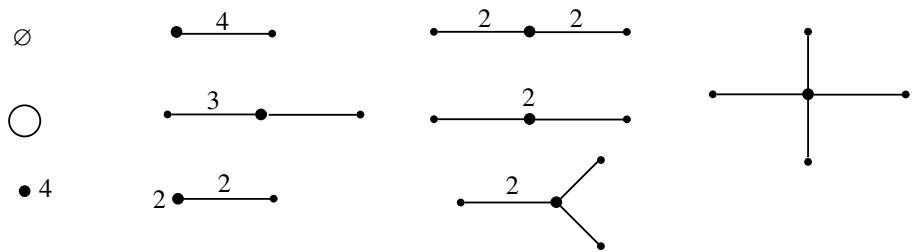 

Finally, we assign this data associated to $M_i$ to the graph $\gL_i$ as 
follows.  The data of $S_{i j}$ will be attached to its $S$--vertex, the graph 
$\gP_{i k}$ (or its reduced graph $\tilde \gP_{i k}$) assigned to the 
$Y\!$--node of $\cY_{i k}$, and the topological attaching data will be assigned to the 
edge from $S_{i j}$ to $\cY_{i k}$.  

Next, just as we did for $M$, we compute the topological invariants of each 
irreducible medial component $M_i$.  This time it will be in terms of the 
extended graph $\gL_i$ and the associated data.  
\subsection*{Associated numerical data for irreducible medial components}

The numerical invariants can be divided into those that are invariants of the graph $\gL_j$, 
and those which are given by the associated data attached to the vertices and edges of 
$\gL_j$.  In general, for an invariant $I_j$ of the medial component $M_j$, we let $I = 
\sum_j I_j$, summed over the irreducible medial components of $M$, be the associated global 
invariant of $M$.  The only exception is our definition of $\nu(M)$ given by \eqref{Eqn3.1b}.  
We summarize these invariants in Table 1.

\subsubsection*{Graph-theoretic numerical data of $\gL_j$}

The key invariants associated to $\gL_j$ are:  the number $s_j$ of $S$--vertices,  the 
number $c_j$ of $Y\!$--nodes and the first Betti number $\gl_j$ of $\gL_j$ (as a topological space).  
 
We can express other invariants such as the (reduced) Euler characteristic 
in terms of these.  For example, the $\mbox{ number of edges of } \gL_j = 
s_j + c_j + \gl_j -1$.  We introduce another purely graph-theoretic invariant 
which appears in computing $\chi(M_j)$: 
\begin{equation}
\label{Eqn3.1a}
  \nu(\gL_j) \quad = \quad s_j - c_j - \gl_j. 
\end{equation}
Using this we also introduce the global graph-theoretic invariant of $M$:
\begin{equation}
\label{Eqn3.1b}  \nu(M) \quad = \quad  \sum_j \nu(\gL_j) \, -  \, \gb_1
\end{equation}
where the sum is over the irreducible medial components of $M$ and $\gb_1$ 
is first Betti number of $\gG(M)$ (as a topological space and $= -\tilde 
\chi(\gG(M))$).
\subsubsection*{Numerical associated data of $\gL_j$}
Second, we list the numerical data defined from the associated data. First of all, an 
$S$--vertex is associated to $S_{j k}$, which is a compact surface with boundary.  The 
boundary will have connected components which are either edge curves or are attached to the 
$Y\!$--network $\cY_j$.  We let $e_{j k}$ denote the number which are medial edge curves of 
$M_j$. For each connected component $\cY_{j k}$ of the $Y\!$--network, we let $v_{j k}$ denote 
the number of vertices in $\gP_{j k}$.  This is the number of $6$- -junction points in the 
$Y\!$--network $\cY_j$.  

\begin{table}[ht!]
\label{Table1} 
\begin{tabular}{lll}
Invariants of   &  Global invariants   &   Description of medial \\ 
 medial components   & of medial axis      &   component invariants  \\[1ex]
Graph invariants   &      &   \\ [1ex]
$s_j$    &  $s$   &  number $S$--vertices of $\gL_j$ =    \\
       &       &   number medial sheets of $M_j$    \\[0.5ex]
$c_j$  &  $c$  &  number $Y\!$--nodes of $\gL_j$  \\
           &             &   = number connected   \\
           &            &     components of $\cY_j$  \\[0.5ex]
$\gl_j$  &  $\gl$   &  $\mathrm{rk}\, H_1(\gL_j)$  \\[0.5ex]
  $\nu(M_j)$  & $\nu(M)$    &   $\nu(M_j)$ given by \eqref{Eqn3.1a}; \\
      &        & $\nu(M)$ given by \eqref{Eqn3.1b} \\[1.5ex]
Graph data  &      &   \\
invariants  &      &   \\ [1ex]
$v_j$    &  $v$    &  number vertices of $\Pi_j$  =    \\
       &       &   number $6$--junction points of $M_j$    \\[0.5ex]
$e_j$    &  $e$     &  number medial edge curves of $M_j$  \\[0.5ex]
$G_j$    &  $G$   &  total weighted genus of  $M_j$ \eqref{Eqn3.2d} \\[0.5ex]
$q_j$    &  $q$    &  sum of $q_{j k}$ given by \eqref{Eqn3.2b} \\
$Q_j$    &  $Q$   &   first Betti number of associated  \\
       &        &    $1$--complex $ \mathcal{S}\cY_j^{\prime} \subset M_j$      \\
       &        &        
\end{tabular} 
\centerline{Table 1:  Invariants $I_j$ of Medial Components $M_j$ and their 
Global Versions} 
\centerline{ $I$ for the medial axis, where $I = \sum_j I_j$.}
\end{table}

Next, $S_{j k}$ has a genus in the sense that if we attach a $2$--disk to each boundary 
component, then we obtain a compact surface (without boundary).  By the classification of 
surfaces, if it is orientable, it is homeomorphic to a connected sum of tori and the number 
is its genus; while if the surface is nonorientable, then it is a connected sum of real 
projective planes, and that number is the genus.  In either case we denote the genus by $g_{j 
k}$, and refer to it as the genus of $S_{j k}$.  We define the \textit{weighted genus of 
$S_{j k}$} to be 
\begin{equation}
\label{Eqn3.2a}
 \tilde g_{j k} \quad = \quad \Bigg\{ 
\begin{array}{ll} 
2g_{j k} &\mbox{ if $S_{j k}$ is orientable}  \\ 
g_{j k}  &\mbox{ if $S_{j k}$ is nonorientable} \end{array}  
\end{equation}
Then the \textit{total weighted genus of $M_j$}
\begin{equation}
\label{Eqn3.2d}
 G_j \, \, \overset{\text{def}}{=} \,\, \sum_k \tilde g_{j k} \quad   \mbox{ summed 
over all medial sheets  $S_{j k}$ of $M_j$}. 
\end{equation}
From the weighted genus, we define $q_{j k}$ as follows.  
\begin{equation}
\label{Eqn3.2b}
 q_{j k} \quad = \quad \Bigg\{ 
\begin{array}{ll} 
\tilde g_{j k} + e_{j k} -1  & \mbox{  if  } e_{j k} > 0  \\ 
\tilde g_{j k}  & \mbox{  if  } e_{j k} = 0 \end{array}  
\end{equation}
Note this is the first Betti number of a $1$--complex in $S_{j k}$ which will 
be used in Sections \ref{S:sec6} to \ref{S:sec8} to compute the homology and 
fundamental group of $M_j$.
Then we define 
$$ q_j = \sum q_{j k} \quad \mbox{summed over all medial sheets $S_{j k}$ 
in $M_j$.}   $$ Although $s_j$, the number of medial sheets in $M_j$, is a graph theoretic 
invariant, there are subsets of $S$--vertices whose definition depends on associated data.  
Of these the most important is $s_{0 j}$, \textit{the number of sheets without edge curves}.  
This is a sum $s_{0 o, j} + s_{0 n, j}$, where $s_{0 o, j}$, resp.\ $s_{0 n, j}$, denotes the 
number of orientable, resp.\ nonorientable, sheets without medial edge curves.  

Then by \eqref{Eqn3.2a}, \eqref{Eqn3.2d} and \eqref{Eqn3.2b}, there is the 
following relation between these invariants:
\begin{equation}
\label{Eqn3.2c}
 q_j \, \, = \,\, G_j \, + \,  e_j \, -\, (s_j - s_{0 j})  
\end{equation}

\subsection*{Topological invariants of irreducible medial components}
We next describe how to determine the Euler characteristic, homology, and fundamental group 
of $M_j$.  In \fullref{S:sec8}, we introduce a \lq\lq minimal $1$--complex\rq\rq\ 
$\mathcal{S}\cY_j^{\prime}$ constructed from the singular data of $M_j$.  It consists of three 
contributions:  the graph structure $\gL_j$, the $Y\!$--network $\cY_j$, and a $1$--skeleton 
from the  medial sheets.  The $\mathrm{rk} H_1(\mathcal{S}\cY_j^{\prime})$  is given by $Q_j = \gl_j + (v_j + 
c_j) + q_j$, which is the sum of the three contributions.  Then in \fullref{S:sec8}, we 
define an \textit{algebraic attaching homomorphism}
$$  \Psi_j \co \Z^{s_{0 j}} \to \Z^{Q_j}.$$
  From this homomorphism we can compute the homology of $M_j$ in terms 
of the singular data.  Second, we shall also define in \fullref{S:sec7} a set of 
generators and relations to compute the fundamental group of $M_j$.  These 
provide two ingredients for the topology of $M_j$.  
\begin{Thm}
\label{Thm3.3}
Suppose that $M_j$ is an irreducible medial component.  Then there are the 
following properties in homology:
\begin{enumerate}
\item $M_j$ has torsion free homology.
\item  The reduced Euler characteristic of $M_j$ is given by
\begin{equation*}
\tilde \chi(M_j)  \quad  = \quad s_{0 j} - Q_j.
\end{equation*}
Alternately, it can be written in terms of the extended graph $\gL_j$, and 
associated data of $\gL_j$:  
\begin{equation}
\label{Eqn3.3b}
\begin{split}
\tilde \chi(M_j)  \quad  &= \quad \nu (\gL_j)\,  - \, ( G_j + e_j +  v_j)   
 \\
&= \quad s_j  - \, (e_j +  v_j +  c_j +  G_j + \gl_j).
\end{split}
\end{equation}
\item  The homology groups $H_2(M_j)$, respectively $H_1(M_j)$, are the 
kernel and cokernel of the algebraic attaching homomorphism $\Psi_j$; see 
\eqref{Eqn8.11}.  Furthermore, there are bounds
\begin{equation*}
\mbox{rk}\,(H_2(M_j; \Z )) \leq s_{0 o, j} \quad \mbox{ and } \quad q_j \leq 
\mbox{rk}\,(H_1(M_j; \Z )) \leq Q_j - s_{0 n, j}.
\end{equation*} 
In addition, there are the properties for the fundamental group:
\item There is a continuous map $\psi_j \co M_j \to \gL_j$ such that the 
induced map $\pi_1(M_j) \to \pi_1(\gL_j)$  is surjective.
\item  The fundamental group has a presentation by generators and 
relations 
$$ \pi_1(M_j) \quad \simeq \quad F_{Q_j}/ \langle \{r_{j i} \} \rangle $$
where the relations $r_{j i}$ are given by \eqref{Eqn7.5b}.
\end{enumerate}
\end{Thm}

\begin{proof}
For (1), it is sufficient to show $H_*(M)$ is torsion free; for $H_*(M_j)$ is a 
direct summand by \fullref{Prop2.4}.  Then since $H_*(M) \simeq 
H_*(\Omega_0)$ where $\Omega_0 = \Omega \backslash \cB$, it is sufficient to show 
$H_*(\Omega_0)$ is torsion free.  We may decompose $\R^3 \backslash \cB = 
\Omega_0 \cup \Omega^{\prime}$ as a disjoint union.  Thus, $H_*(\Omega_0)$ is a 
direct summand of $H_*(\R^3 \backslash \cB)$.  However, by Alexander 
duality $H_j(\R^3 \backslash \cB) \simeq H^{2-j}(\cB)$ for $j > 0$.  Since 
$\cB$ is a compact orientable surface, $H^*(\cB)$ is torsion free, giving the 
conclusion.  

The derivation of the formula for the Euler characteristic, and the 
computation of the homology in terms of the algebraic attaching 
homomorphism are given in \fullref{S:sec8}.  

The computation of the fundamental group is given in \fullref{S:sec7}, using 
the CW--decom\-position of $M_j$ in terms of the singular structure given in 
\fullref{S:sec6}.  
\end{proof}
These results provide the following restrictions on geometric complexity.  
\begin{Corollary}
\label{Cor3.4}
Consider an irreducible medial component $M_j$.  

First, suppose $H_1(M_j) = 0$. Then we have the following: 
\begin{enumerate}
\item  The graph $\gL_j$ is a tree.
\item  Each medial sheet $S_{j k}$ has genus $0$, ie it is a $2$--disk with a 
finite number of holes.  
\item  At most one of the boundary components of $S_{j k}$ is an edge 
curve of $M_j$.  
\item  Suppose, in addition, $H_2(M_j) = 0$. Then $e_j$ is the number of 
medial sheets with edge curves, and the following \lq\lq Euler Relation\rq\rq\ 
holds between the number of sheets without edge curves (LHS) and the 
(RHS), which is $\mathrm{rk}\, H_1(\cY_j)$, a topological invariant of the $Y\!$--network 
\begin{equation}
\label{Eqn3.4}
         s_j \, - \, e_j \quad = \quad v_j  \, + \, c_j.
\end{equation}  
\item  If $M_j$ is simply connected (without requiring $H_2(M_j) = 0$), then 
$M_j$ will satisfy the fundamental group relation given in  \fullref{Thm5.0}.  
\item  Finally, if $M_j$ is contractible, then all of the preceding hold.
\end{enumerate}
\end{Corollary}

The derivation of the Corollary from \fullref{Thm3.3} will be given in 
\fullref{S:sec8}

Finally, we combine the results of \fullref{Thm3.3} with the decomposition in \fullref{Prop2.4} to compute the \textit{homology and reduced Euler characteristic of $M$ from 
the singular invariants measuring geometric complexity}.   We let 
$$\Psi \,\, = \,\, \oplus_j \Psi_j \, \co \Z^{s_0} \to \Z^Q, $$
where $s_0 = \sum_j s_{0 j}$ is the total number of medial sheets in all of 
the $M_j$ without edge curves (and $Q = \sum_j Q_j$).  
\begin{Thm}
\label{Thm3.5}
Let $M$ be the Blum medial axis of a connected region $\Omega$ with smooth 
generic boundary $\cB$.  Then $H_{*}(\Omega) \simeq H_{*}(M)$ have the 
following properties.
\begin{enumerate}
\item
$M$ (and $\Omega$) have torsion free homology.
\item
$ H_2(M; \Z) \simeq \ker (\Psi) \quad \mbox{ and } \quad  H_1(M; \Z) 
\simeq \coker (\Psi) \oplus H_1(\Gamma(M); \Z)$.
\item  $
\mbox{rk}(H_2(M; \Z )) \,\, \leq \,\, s_{0 o} \quad \mbox{ and } \quad 
q + \gb_1 \,\, \leq \,\, \mbox{rk}(H_1(M; \Z )) \,\, \leq \,\,
Q - s_{0 n} + \gb_1 ,
$
where $s_{0 o}$, resp.\ $s_{0 n}$ denote the number of orientable, resp.\ nonorientable, 
medial sheets of $M$ without edge curves and as before $\gb_1 = \mbox{rk}(H_1(\Gamma(M); \Z 
))$.
\item  The reduced Euler characteristic is given by
\begin{equation*}
\tilde \chi(M) \quad=\quad\nu (M)\,  - \, ( G + e +  v)\quad =\quad s - 
\, (e +  v + c + G + \gb)
\end{equation*}
where $\gb = \gb_1 + \sum_j \gl_j$.
\item
In the case $\Omega$ (and hence $M$) is contractible, $G = 0$, $\gb = 0$, and 
we have the relation
\begin{equation*}
s - e \quad = \quad c +  v.
\end{equation*}
\end{enumerate}
\end{Thm}
\begin{proof}
The proof follows from $M$ being a strong deformation retract of $\Omega$ 
(see eg \cite{D1}), using the sum formula given in \fullref{Prop2.4}, and substituting the formulas for $H_i(M_j; \Z )$ and $\tilde 
\chi (M_j)$ given in \fullref{Thm3.3}. 
\end{proof}
\begin{Example}
\label{Ex3.6}   The simplest irreducible medial component $M_i$ has $\cY_i = 
\emptyset$.  Then $M_i$ is a compact connected surface with boundary
so $s_i = 1$ and $e_i = $ number of boundary components.  The
$Y\!$--network invariants $c_i = v_i = 0$ and the graph $\gL_i$
consists of a single $S$--vertex so $\gl_i = 0$.  Then $G_i = Q_i =
q_i$ is the rank of $H_1(M_i)$, and $\nu (M_i) = 1$.  If $e_i = 0$
then $M_i$ must be orientable, $s_{0 i} = 1$, and $\Psi_i$ is the zero
map; while if $e_i > 0$ then $s_{0 i} = 0$ and $\Psi_i$ is again the
zero map.
\end{Example}
\begin{Example}
\label{Ex3.7} 
The next simplest case has $\cY_i = \circ$.  Already there arises a question 
of which nonintersecting embedded surfaces can be the medial sheets.  The 
easiest case is when the medial sheets are $2$--disks and annuli (examples in 
\fullref{fig.1}, \fullref{fig.5} and the first three examples of \fullref{fig.9}).  However, an example exists of a single medial sheet which is 
annulus with a single medial edge curve to form a \lq\lq collapsed version of 
the umbilic bracelet\rq\rq\ as in \fullref{fig.13a}.  
\end{Example}

\section{General structure theorem}  
\label{S:sec4} We can now state the full structure theorem for the Blum medial axis.  and in 
the next section state the form it takes for contractible $\Omega$.  

\begin{Thm}[General structure theorem]
\label{Thm4.1}  Suppose $\Omega \subset \R^3$ is a compact connected region with generic smooth 
boundary $\cB$ and Blum medial axis $M$.  Then  there is associated to $\Omega$ a two level 
graph structure.  
\begin{enumerate}
\item  At the top level $\Gamma(M)$ is a directed extended graph.  It 
consists of vertices corresponding to the irreducible medial components $M_i$ of $M$ and 
directed edges from $M_i$ to $M_j$ corresponding to each connection attaching an edge of 
$M_i$ to $M_j$ along a fin curve.   
\item   To each $M_j$ is associated a graph $\gL_j$ with two types of 
vertices: $S$--vertices $\scriptstyle\blacksquare$ corresponding to smooth medial sheets $S_{j k}$ of 
$M_j$, and nodes $\bullet$ corresponding to the connected components $\cY_{j k^{\prime}}$ of 
the $Y\!$--network of $M_j$.  There is an edge from $S_{j k}$ to $\cY_{j k^{\prime}}$ for each 
boundary $S^1$ of $S_{j k}$ which is attached to $\cY_{j k^{\prime}}$.  
\item  Each medial sheet $S_{j k}$ is compact surface with boundary 
(possibly nonorientable).  Each $Y\!$--network component $\smash{\cY_{j k^{\prime}}}$ 
can be described as a $4$--valent extended graph $\smash{\gP_{j k^{\prime}}}$ (or by 
the corresponding reduced graph $\smash{\tilde \gP_{j k^{\prime}}}$).  
\item   At the third level, the extended graph $\gL_j$ has data assigned to 
the $S$--vertices, $Y\!$--nodes, and edges.  To an $S$--vertex for $S_{j k}$ is assigned the 
data $(g, o, e)$ indicating the genus $g$, $o = 1$ or $0$ denoting orientability or 
nonorientability and $e$ the number of Blum edge curves. To a $Y\!$--node for $\smash{\cY_{j 
k^{\prime}}}$ is assigned the $4$--valent extended graph $\smash{\gP_{j k^{\prime}}}$. To an edge 
from $S_{j k}$ to $\smash{\cY_{j k^{\prime}}}$, the topological attaching data of the boundary circle 
of $S_{j k}$ to $\smash{\cY_{j k^{\prime}}}$.  
\item   Furthermore, the graph theoretic data of each $\gL_j$, the number 
of edge curves, $6$--junction points and the weighted total genus satisfy the 
equation  \eqref{Eqn3.3b}.  
\item The fundamental groups of each irreducible medial component are 
given in terms of generators and relations by \eqref{Eqn7.5a}, and the 
homology groups of each irreducible medial component are torsion free and 
given as the kernel and cokernel of the algebraic attaching map 
\eqref{Eqn8.11}.  
\end{enumerate}
\end{Thm}
\begin{proof}
We have established the decomposition into irreducible medial components, 
which in turn can be represented by the attaching of medial sheets to the 
$Y\!$--network.  

The computation of the fundamental group in terms of generators and 
relations will be given in \fullref{S:sec7} based on the CW--decomposition 
given in \fullref{S:sec6}, and the computation of both the homology and the 
reduced Euler characteristic will be given in \fullref{S:sec8}.
\end{proof}
Several examples we have already seen illustrate the general structure 
theorem.
\begin{Example}
\label{Ex4.3}   In \fullref{fig.8c}, $M$ is formed from $4$ medial components, of 
which $M_2$, $M_3$, and $M_4$ are topologically $2$--disks ($\cY =
\emptyset$ so $c_i = v_i = 0$, and $s_i = e_i = 1$).  If we slide
$M_4$ onto $M_1$, then the graph $\gG (M)$ is a tree.  $M_1$ is formed
from $5$ medial sheets, all of genus $0$: two are $2$--disks, $2$ are
annuli, and one is a $2$--disk with two holes.  Except for the annuli,
each of the other $3$ medial sheets of $M_1$ have single medial edge
curves.  The $Y\!$--network is empty except for $M_1$ where it
consists of $2$ disjoint $\circ$.  $\gL_1$ is a tree, and the
attaching of the $2$--disks for $M_1$ kill the two free generators for
$\pi_1(\cY_1) \cup \gL_1$.  Hence, $M_1$ is simply- connected.  Also,
we see $\nu(M_1) = 5 -2 -0 = 3$ and the other $\nu(M_i) = 1$.  Hence,
$\tilde \chi(M) = (3 + 3\cdot 1 + 0) -(0 + 3 + 3\cdot 1 + 0) = 0$.
 
Thus, $H_2(M_1) = 0$, and $M_1$ and hence all of the $M_i$ and $M$ are 
contractible.
\end{Example}
\begin{Example}
\label{Ex4.4}  
The medial axis shown in \fullref{fig.13} is irreducible so $M = M_1$.  There 
are five medial sheets ($s = 5$) all of which have genus $0$: $3$ annuli (two 
of which have a medial edge curve) and $2$ disks (so $G = 0$, $e = 2$, $q = -
1$ and the $Y\!$--network consists of $2$ disjoint $\circ$ ($c = 2$, $v = 0$).  
Also, $\gL_j$ is a tree so $\gl = 0$.  As above, the attaching of the $2$-
disks for $M_1$ kill the two free generators for $\pi_1(\cY_1) \cup \gL_1$ 
so $M = M_1$ is simply connected.  We see $\nu(M) = 5 - 2 + 0 = 3$.  Thus, 
$\tilde \chi(M) = 3 - (0 + 2 + 0) = 1$.  Hence, $H_2(M_1) = \Z$, and $M = 
M_1$ is homotopy equivalent to $S^2$.
\end{Example}

\section{Structure theorem for contractible regions}  
\label{S:sec5} For contractible regions $\Omega$, the structure theorem for the Blum medial axis 
takes a considerably simplified form.  We saw in \fullref{Cor3.4} that the medial 
sheets and extended graphs had to have a special form.  As well, we define in \fullref{S:sec7} 
a set of elements $r_{i j}$ in $\pi_1(\mathcal{S}\cY_i^{\prime})$ corresponding to medial sheets 
$S_{i j}$ without edge curves; see \eqref{Eqn7.5b}.  The computation of the fundamental 
group in \fullref{Thm7.5} leads to the following condition.

\begin{TitleThm}\setobjecttype{TitThm}
\label{Thm5.0}
For an irreducible medial component $M_i$, the set of elements 
$$   \{ r_{i j} :  \mbox{ all $j$ for which $S_{i j}$ is without edge curves of 
$M_i$} \} $$
form a set of generators for $\pi_1(\mathcal{S}\cY_i^{\prime})$.  
\end{TitleThm}
Taken together the conditions on the extended graphs being trees, the 
numerical \lq\lq Euler Relation\rq\rq\ and the preceding \fullref{Thm5.0} completely characterize contractible regions.  

\begin{Thm}[Structure theorem for contractible regions]
\label{Thm5.1} Suppose $\Omega \subset \R^3$ is a contractible bounded region with generic 
smooth boundary $\cB$ and Blum medial axis $M$.  Then there is associated to $\Omega$ a 
\textit{multilevel directed tree structure} which determines the simplified structure $\hat 
M$ associated to $M$.  
\begin{enumerate}
\item  At the top level, $\Gamma(M)$ is a directed tree consisting of 
vertices corresponding to the irreducible medial components $M_i$ of $M$, and there are 
directed edges from $M_i$ to $M_j$ corresponding to the attaching of an edge of $M_i$ to 
$M_j$ along a fin curve in $\hat M$.   
\item  At the second level, to each $M_j$ is associated a directed tree 
$\gL_j$ with two types of vertices: $S$--vertices $\scriptstyle{\blacksquare}$ corresponding to smooth 
medial sheets $S_{j k}$ of $M_j$, and $Y\!$--nodes $\bullet$ corresponding to the connected 
components $\cY_{j k^{\prime}}$ of the $Y\!$--network of $M_i$.  There is an edge from $S_{j 
k}$ to $\cY_{j k^{\prime}}$ if a boundary \lq\lq circle\rq\rq\ of $S_{j k}$ is attached to 
$\cY_{j k^{\prime}}$.  Given $S_{j k}$ and $\cY_{j k^{\prime}}$, there is at most one such 
boundary \lq\lq circle\rq\rq.  
\item  Each medial sheet $S_{j k}$ is topologically a $2$--disk with a finite 
number of holes.  At most one of the boundary circles of a medial sheet is an 
edge curve of the $M_j$.  Each $Y\!$--network component $\cY_{j 
k^{\prime}}$ can be described as a $4$--valent extended graph $\gP_{j 
k^{\prime}}$ (or by the corresponding reduced graph $\tilde \gP_{j 
k^{\prime}}$.  
\item  At the third level, the tree $\gL_j$ has data assigned to the 
$S$--vertices, $Y\!$--nodes, and edges.  To an $S$--vertex for $S_{j k}$ is assigned $(h, e)$ 
indicating the number of holes and Blum edge curves. To a $Y\!$--node for $\cY_{j k^{\prime}}$ 
is assigned the $4$--valent extended graph $\gP_{j k^{\prime}}$. To an edge from $S_{j 
k}$ to $\cY_{j k^{\prime}}$, the topological attaching data of the single boundary circle of 
$S_{j k}$ to $\cY_{j k^{\prime}}$.  
\item  Furthermore, for an irreducible medial component $M_j$, the number 
of medial sheets, edge curves, $6$--junction points and the total number of 
connected components of the $Y\!$--network $\cY_{i}$ satisfy the Euler 
relation \eqref{Eqn3.4}, and each $M_j$ satisfies the fundamental group 
relation \fullref{Thm5.0}.
\end{enumerate} 
Conversely, suppose we are given a bounded region $\Omega$ in $\R^3$ with 
smooth generic boundary and Blum medial axis $M$ so that: the top level 
graph $\Gamma(M)$ is a tree; for each irreducible medial component $M_j$, 
the graph $\gL_j$ is a tree; the medial sheets are topologically $2$--disks 
with a finite number of holes, having at most one boundary circle an edge 
curve of $M_j$; the numerical invariants of $M_j$ satisfy \eqref{Eqn3.4}; and 
each $M_j$ satisfies \fullref{Thm5.0}.  Then 
$M$ and (hence) $\Omega$ are contractible.  
\end{Thm}
\begin{proof}
If $\Omega$ is contractible, then so is $M$ which is a strong deformation retract of $\Omega$.  
Then by \fullref{Cor2.6}, $\Gamma(M)$ is a tree and each $M_j$ is contractible.  Then 
by \fullref{Cor3.4}, each $\gL_j$ is a tree, and each medial sheet $S_{j k}$ is 
topologically a $2$--disk (ie has genus $0$), with a finite number of holes.  At most one of 
the boundary circles of a medial sheet is an edge curve of the $M$.  

 Furthermore, 
$\tilde \chi(M_j) = 0$, yielding the Euler relation \eqref{Eqn3.4}.  As $\pi_1(M_j) = 0$, by 
\fullref{Thm7.5}, \fullref{Thm5.0} must hold.  

 Conversely, by \fullref{Thm7.5}, \fullref{Thm5.0} implies that $\pi_1(M_j) = 0$ for each $j$.  
Hence, $H_1(M_j; \Z) = 0$.  As the graphs $\gL_j$ are trees, $\gl_j = 0$, and all sheets have 
$q_j = 0$.  Thus, the Euler condition implies by \eqref{Eqn3.3b} that $\tilde \chi 
(M_j) = 0$.  Hence, as $H_2(M_j; \Z)$ is torsion free, $H_2(M_j; \Z) = 0$.  Thus, as 
$\Gamma(M)$ is a tree, by \fullref{Prop2.4}, $M$ is simply connected and $\tilde 
H_*(M; \Z) = 0$.  As $M$ is a CW--complex, a theorem of Hurewicz implies $M$ is contractible.
\end{proof}

\begin{Example}[Simplest examples]
\label{Ex5.3}
 As the simplest examples illustrating the structure theorem, we consider 
those for which the second level trees are $\Pi_j = \emptyset$ or $\circ$. 

(1)\qua \lq\lq Simple contractible examples\rq\rq\qua   Each $\Pi_j = \emptyset$ so each 
$M_i$ is topologically a $2$--disk.  This is the simplest type of region.  For example, in 
\fullref{fig.7b} and \fullref{fig.7a}, the medial axis is decomposed into irreducible medial 
components of this type.  In computer imaging, the $M$--rep structure of Pizer and coauthors 
\cite{P} is based on the region having such a simple medial structure.  

(2)\qua  The second 
type of regions would allow both $\Pi_j = \emptyset$ or $\circ$ and the only type of medial 
sheets are $2$--disks and annuli.  Besides the example in \fullref{fig.1}, we also see the 
example in \fullref{fig.5} and the first three examples of \fullref{fig.9} are of this type  

(3)\qua  \fullref{fig.13a} is a collapsed version of the \lq\lq umbilic torus\rq\rq\ 
investigated independently by Helaman Ferguson (who has produced sculptures of it) and 
Christopher Zeeman who called it the \lq\lq umbilic bracelet\rq\rq.  The discriminant for 
real cubic binary forms is a cone on this space.  The \lq\lq umbilic torus/bracelet\rq\rq\ is 
obtained by rotating the hypocycloid of three cusps about an external axis in its plane, 
while rotating by $\frac{2\pi}{3}$.  If we collapse the hypocycloid onto the three line 
skeleton as in \fullref{fig.13a}\,(a), and rotate then we obtain \fullref{fig.13a}\,(b).  It 
consists of an annulus attached to an inner $Y\!$--junction curve and an edge curve.  It has 
$\Pi_j = \circ$ and a single edge curve, and illustrates how only a single medial sheet may 
meet the $Y\!$--junction curve.   

(4)\qua  The next case would have a reduced network graph of 
the form \lq\lq $\bullet\ 4$\rq\rq\ as for the fourth example of \fullref{fig.9}, which is 
contractible.  

It appears to be a rather difficult question to determine exactly when 
examples such as (3) or (4) can be included as part of larger medial structures with special 
properties such as being contractible.  
\end{Example}

\begin{figure}[ht!]
\labellist
\small\hair 2pt
\pinlabel {(a)} [tr] at 96 541 
\pinlabel {(b)} [tr] at 284 541
\endlabellist
\centerline{\includegraphics[width=9cm]{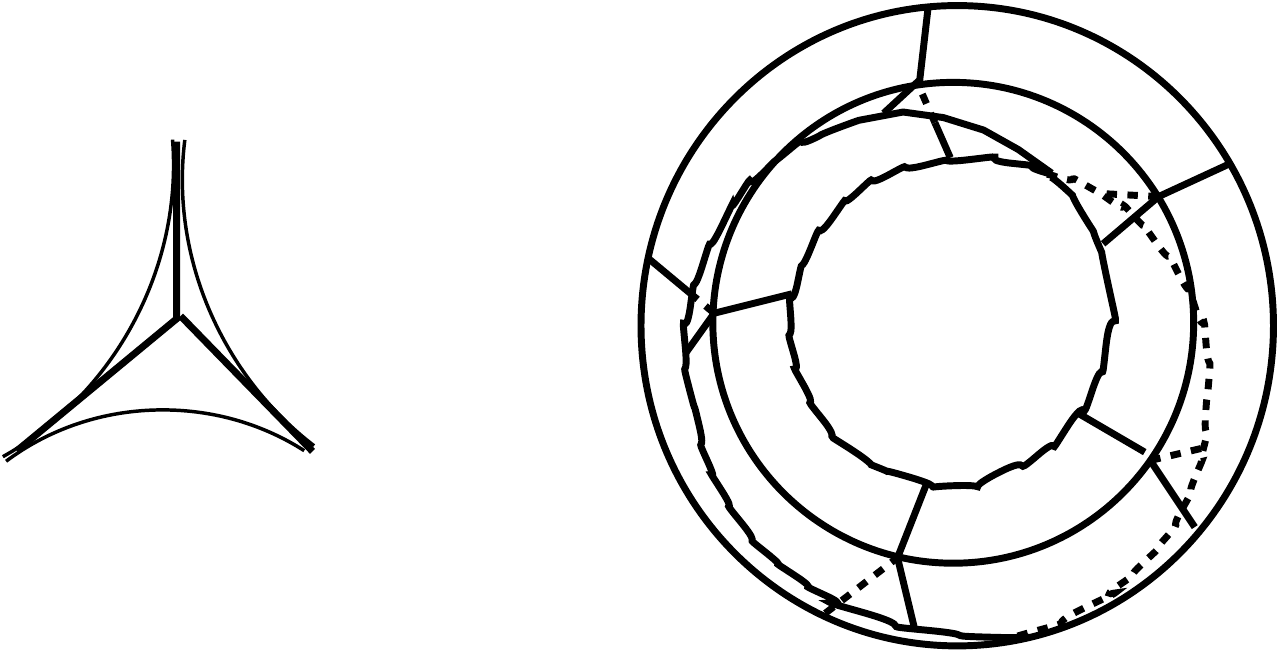}}
\caption{\label{fig.13a} (b) Skeleton of an \lq\lq umbilic 
torus/bracelet\rq\rq, obtained by rotating the three intersecting lines in (a) 
around an axis while rotating by $\frac{2\pi}{3}$}
\end{figure} 

\begin{Example}[Tree structure for a noncontractible region]
\label{Ex5.4}
A region can have tree structures at each of the two levels and consist of 
medial sheets of genus $0$ attached as prescribed by \fullref{Thm5.1} yet not be contractible.  An example is given by \fullref{fig.13}.  $M$ is not contractible, and it is only the Euler relation that 
fails: $s = 5$, $e= 2$, while $m = 2$ and $v = 0$, so $s - e = 3 \neq 2 = v + 
m$.  
\end{Example}

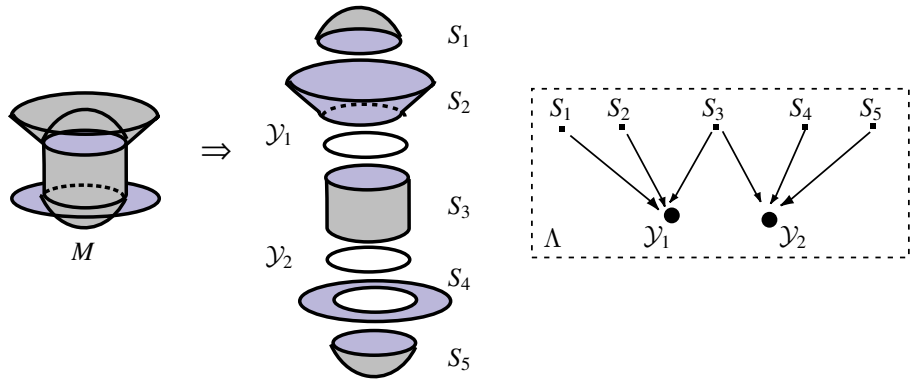
\begin{figure}[ht!]
\labellist
\small\hair 2pt
\pinlabel $M$ [t] at 95 609
\pinlabel $\cY_1$ [r] at 218 668
\pinlabel $\cY_2$ [r] at 218 598
\pinlabel $S_1$ [l] at 300 724
\pinlabel $S_2$ [l] at 300 686
\pinlabel $S_3$ [l] at 300 626
\pinlabel $S_4$ [l] at 300 585
\pinlabel $S_5$ [l] at 300 539
\pinlabel $\Lambda$ at 363 609
\pinlabel $\cY_1$ [tr] at 433 618
\pinlabel $\cY_2$ [tl] at 490 618
\pinlabel $S_1$ [b] at 368 673
\pinlabel $S_2$ [b] at 401 673
\pinlabel $S_3$ [b] at 455 673
\pinlabel $S_4$ [b] at 505 673
\pinlabel $S_5$ [b] at 545 673
\endlabellist
\centerline{\includegraphics[width=12cm]{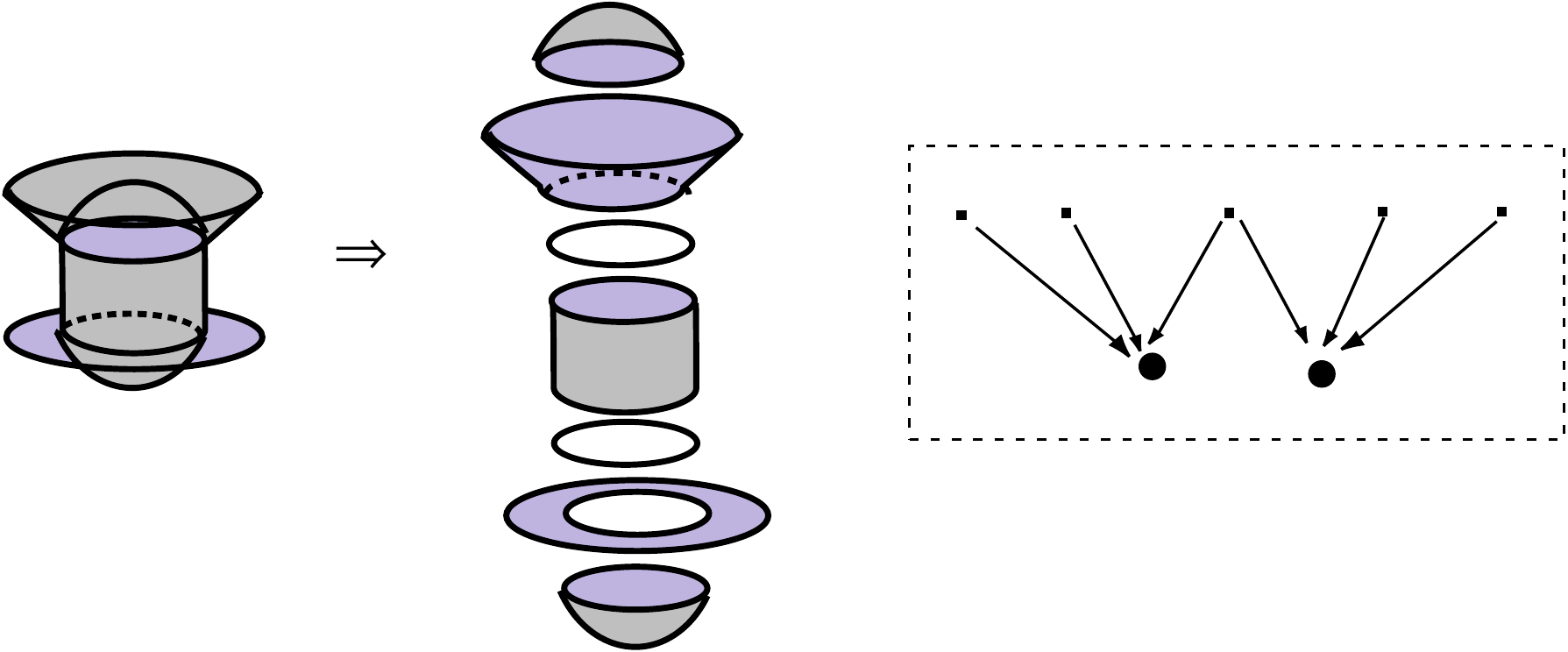}}
\caption{\label{fig.13} Medial axis $M$ of a noncontractible region with tree 
structures for $\Gamma(M)$ and $\gL$}
\end{figure} 

\section{Representations of the topological structure of the medial sheets}  
\label{S:sec6} To determine the topology of an irreducible medial component $M_j$, we give 
two alternate representations of it as either built up by attaching surfaces with boundaries 
or by attaching cells.  These give two alternate useful ways to decompose $M_i$ to compute 
homology and the fundamental group.  We introduce two  subspaces $\mathcal{S}\cY_j \subset 
\mathcal{S}\cY_j^{\prime} \subset M_j$, such that $\mathcal{S}\cY_j$ will contain both $\cY_j$ and a subspace 
homotopy equivalent to the extended graph $\gL_j$.  

 To do this we consider a medial 
sheet $S_{j k}$ of $M_j$.  We may use the standard representation for compact surfaces as 
quotients of the $2$--disk $D^2$ after appropriately identifying edges of the boundary as in 
Massey \cite{Mas}.  In our case, $S_{j k}$ is a quotient of a $2$--disk $D^2$ with a finite 
number of holes obtained by removing the interiors $B_i$ of embedded $2$--disks $\bar B_i, i 
= 1, \dots \ell_{j k} + e_{j k}$.  Here $\ell_{j k}$ is the number of the disk boundaries 
which will be attached to $\cY_j$, and the other $e_{j k}$ will remain edge curves in $M_j$.  
We may assume that for example, the edges of the boundary of $D^2$ are identified as given by 
the classification of surfaces; see eg \fullref{fig.9a}\,(a).  

\begin{figure}[ht!]
\labellist
\small\hair 2pt
\pinlabel {(a)} [tr] at 70 574
\pinlabel {$a_{i_3}$} [r] at 72 651
\pinlabel {$a_{i_2}$} [br] at 79 685
\pinlabel {$a_{i_1}$} [b] at 101 712
\pinlabel {$a_{i_m}$} [bl] at 188 717
\pinlabel $B_1$ [b] at 109 683
\pinlabel {$B_{e_{jk}}$} [l] at 133 588
\pinlabel {$S_{jk}$}  at 185 593
\pinlabel $D'$ at 179 697
\pinlabel {(b)} [tr] at 290 574
\pinlabel {$x_{k0}$} [r] at 359 642
\pinlabel {$S_{jk}$} at 389 589
\pinlabel {$\Delta_{jk}$} at 401 695
\pinlabel {$x_{k0}$} [r] at 499 644
\pinlabel {$\Gamma_{jk}$} [b] at 520 675
\endlabellist
\centerline{\quad\includegraphics[width=12cm]{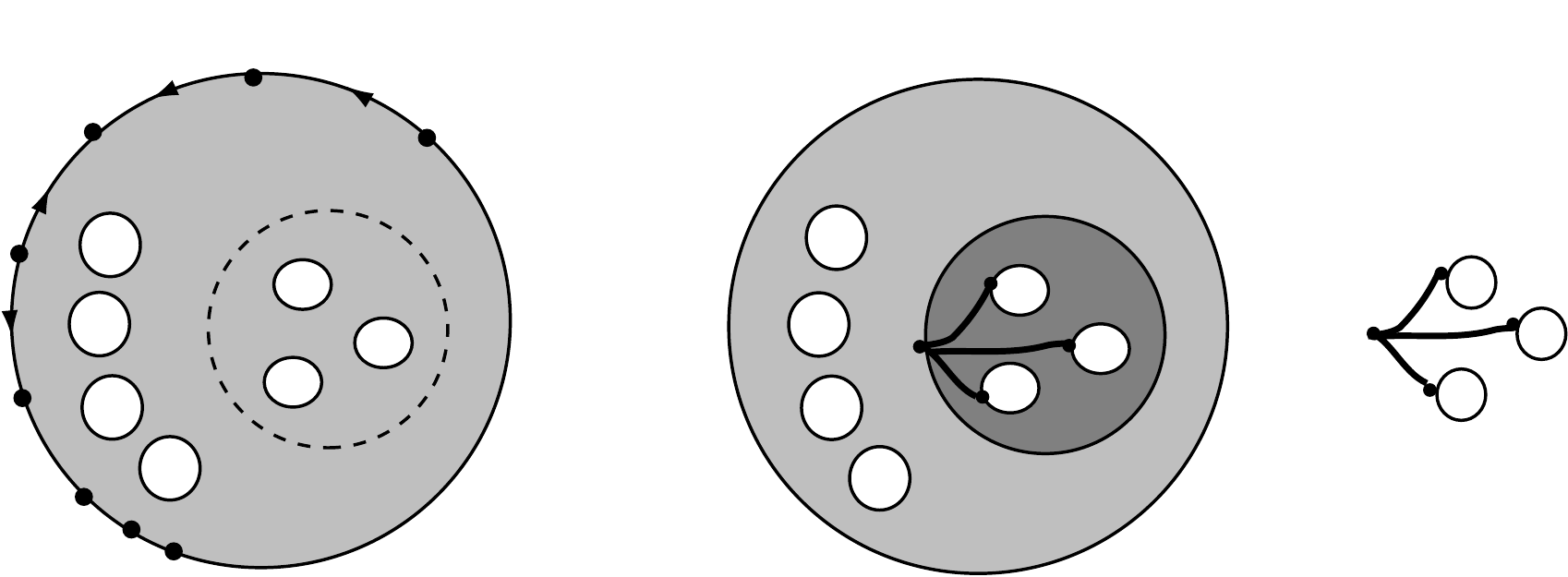}}
\caption{\label{fig.9a}\,(a) $S_{j k}$ identified as the quotient of a 
$2$--disk with a finite number of holes by identifying the edges of the 
boundary;  (b)  Edges of holes to be attached to $\cY$ contained in $\gD$, and 
deformation retract $\Gamma_{j k}$.}
\end{figure}
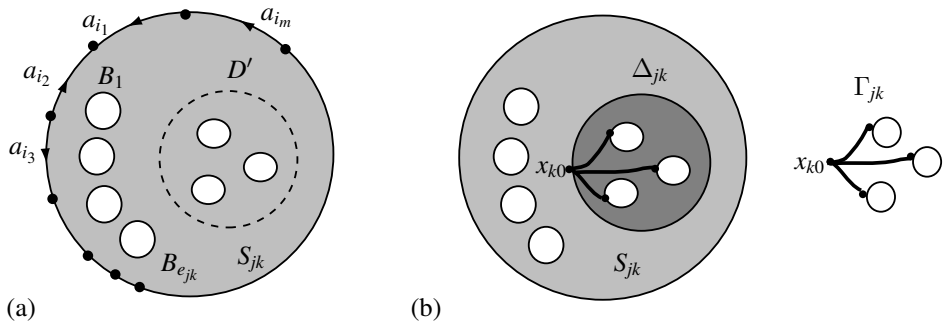 

Thus, $S_{j k} = (D^2 \backslash (\tsty\bigcup_i B_i))/ \sim$.  To be specific, we 
suppose that the boundaries of $\bar B_i, i = 1, \dots \ell_{j k}$ are the 
ones which will be attached to $\cY_j$.  Up to isotopy, we may assume that 
$\bar B_i, i = 1, \dots \ell_{j k}$ are contained in a $2$--disk $D^{\prime}$  
as shown in \fullref{fig.9a}\,(b).  We let $\gD_{j k} = D^{\prime} \backslash 
(\tsty\bigcup_{i = 1}^{\ell_{j k}} B_i)$.  We may construct curves $\gb_i$ in $\gD_{j 
k}$ from a point $x_{j 0}$ on the boundary of $\gD_{j k}$ to boundary points 
$x_{j i}$ of each $B_i$, so the curves are disjoint from each other except at 
$x_{j 0}$.  

We let $\Gamma_{j k}$ denote the union of the curves $\gb_i$ and the union 
of the boundary edges $\partial \bar B_i, i = 1, \dots \ell_{j k}$.  Then 
$\Gamma_{j k}$ is a strong deformation retract of $\gD_{j k}$.  
We let 
$$\mathcal{S}\cY_j \,\, = \,\, (\cY_j \cup (\tsty\bigcup_k \gD_{j k}))/ \sim \quad \mbox{ 
and  } \Gamma_j \,\, = \,\, (\tsty\bigcup_k \Gamma_{j k})/ \sim$$
where $\sim$ denotes the equivalence for attaching the interior boundary 
edges in the $\Gamma_{j k}$ to $\cY_j$.  Since every edge of $\cY_j$ is the 
image of one of the attached edges of some $S_{j k}$, we see that $\cY_j 
\subseteq \Gamma_j \subset \mathcal{S}\cY_j$. Also, $\Gamma_j$ is a strong 
deformation retract of $\mathcal{S}\cY_j$ by retracting each $\gD_{j k}$ to 
$\Gamma_{j k}$ leaving $\cY_j$ fixed.  
$$    \tilde S_{j k} \,\, = \,\, S_{j k} \backslash \mbox{int}(\gD_{j k}).  
\leqno{\hbox{We next let}}
$$  
Then we can attach each $\tilde S_{j k}$ to $\mathcal{S}\cY_j$ along $C_{j k}$  the 
outer boundary edge of $\gD_{j k}$.  Attaching all such $\smash{\tilde S_{j k}}$ 
recovers $M_j$.  We summarize these statements with the following lemma.

\begin{Proposition}
\label{Prop6.1}
Let $M_j$ be an irreducible medial component.  Then the following hold:
\begin{enumerate}
\item  There is a subspace $\Gamma_j \subset \mathcal{S}\cY_j$ of $M_j$ which 
contains the $Y\!$--network $\cY_j$. 
\item  $\Gamma_j$ is a $1$--complex which is a strong deformation retract 
of $\mathcal{S}\cY_j$.  
\item  If we collapse the connected components of $\cY_j$ in $\Gamma_j$ to 
points, then the resulting quotient space is homeomorphic to the second 
level extended graph $\gL_j$. 
\item  There are compact surfaces with boundaries $\tilde S_{j k}$  obtained 
from the medial sheets $S_{j k}$ by removing $\gD_{j k}$, a $2$--disk with 
holes whose edges are identified to $\cY_j$.  Attaching the $\smash{\tilde S_{j k}}$ 
to $\smash{\mathcal{S}\cY_j}$ recovers $M_j$, up to homeomorphism.
\end{enumerate}
\end{Proposition}
We shall refer to this as the \textit{modified medial sheet representation}.
\subsection*{CW--decomposition of $M_j$}
We further simplify the preceding, by introducing an intermediate space $\mathcal{S}\cY_j^{\prime}$. 

 We divide the medial sheets into two types: those with edge curves and those without.  
If $S_{j k}$ has an edge curve, then in the earlier notation it has $e_{j k} > 0$ edge curves 
(each surrounding a hole) in the $2$--disk.  We denote these by $E_{j k}$.  We can choose 
$e_{j k} - 1$ disjoint paths $\tau_{k i}$ from $x_{k 0}$ to points on the distinct edge curves 
$E_{k i}$ for $i < e_{j k}$.  We can then form $\tilde E_{k i} = E_{k i} \cup \tau_{k i}$.   
Likewise, we choose a curve $\tau_{k 0}$ from $x_{k 0}$ to the boundary of $D^2$, not 
intersecting the paths $\tau_{k i}$.  We can form loops $\sigma_{j i}$ centered at $x_{k 0}$ by 
following $\tau_{k i}$, then going around the $i$--th edge curve counterclockwise, and then 
following $\tau_{k i}$ backwards back to $x_{k 0}$ (see \fullref{fig.10a}).  
\begin{figure}[ht!]
\labellist
\small\hair 2pt
\pinlabel {$\sigma_{j1}$} [r] at 74 662
\pinlabel {$\tau_{k0}$} at 110 686
\pinlabel {$\sigma_{je_{jk^{-1}}}$} at 118 600 
\pinlabel {$\tilde{S}_{jk}$} at 149 585
\pinlabel {$\delta_{jk}$} [b] at 156 679
\pinlabel {$\zeta_{jk}$} [bl] at 180 696
\pinlabel {$\sigma_{j1}$} [r] at 271 644
\pinlabel {$\sigma_{je_{jk^{-1}}}$} at 318 578 
\pinlabel {$\tau_{k0}$} at 311 665
\pinlabel {$\delta_{jk}$} [bl] at 363 654
\pinlabel {$L_{jk}$} at 341 697
\pinlabel {$\sigma_{j1}$} [r] at 461 641
\pinlabel {$\tau_{k0}$} at 500 658
\pinlabel {$\sigma_{je_{jk^{-1}}}$} at 504 576 
\pinlabel {$\tilde{L}_{jk}$} at 538 691  
\endlabellist
\centerline{\includegraphics[width=12cm]{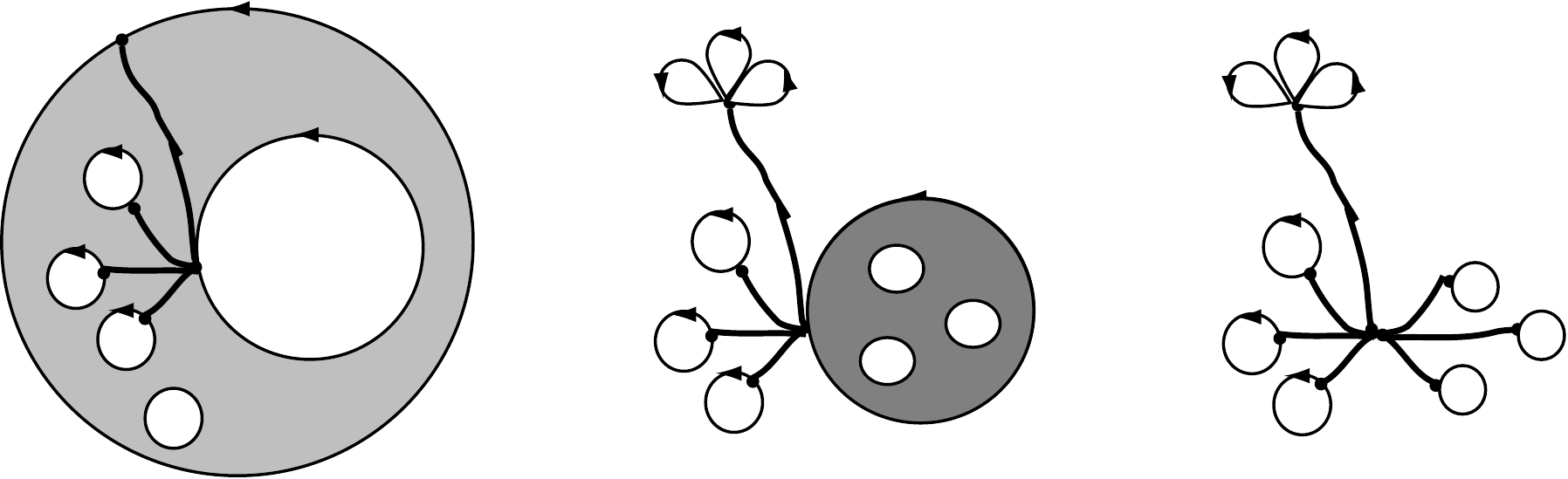}}
\caption{\label{fig.10a} The spaces (a) $\tilde S_{j k}$, (b) $L_{j k}$ and (c) 
$\tilde L_{j k}$}
\end{figure}
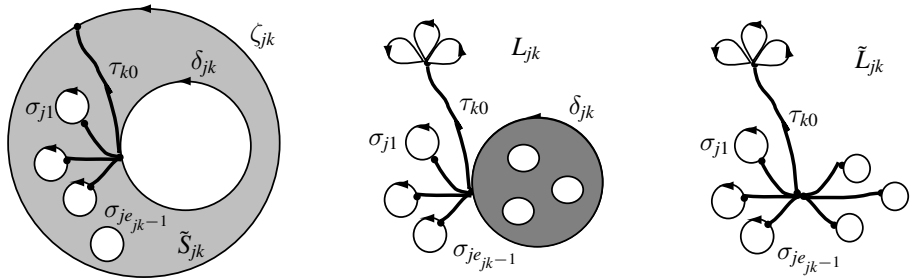 
$$ L_{j k}\,\,  = \,\,  \gD_{j k} \underset{k < e_{j k}}{\cup} \tilde E_{j k} 
\cup \tau_{k 0} \cup (\partial D^2/\sim )    \leqno{\hbox{We let}}$$
and $\tilde L_{j k}$ is defined as $ L_{j k}$ except we replace $\gD_{j k}$ by 
$\Gamma_{j k}$.
Here \lq\lq $\sim$\rq\rq\ denotes the identification on the outer boundary of 
the $2$--disk.  In the case that $S_{j k}$ has no edge curves of $M_j$, then 
$$ L_{j k}\,\,  = \,\,  \gD_{j k} \cup \tau_{k 0} \cup (\partial D^2/\sim ). $$
We define
\begin{equation}
\label{Eqn6.3}
 \mathcal{S}\cY_j^{\prime} \,\, = \,\, \mathcal{S}\cY_j \cup_k L_{j k} \mbox{where the 
union is over all medial sheets in $M_j$.} 
\end{equation}
We describe the relation between $M_j$ and $\smash{\mathcal{S}\cY_j^{\prime}}$ in two steps.  First, we 
relate $\smash{S_{j k}}$ and $\smash{L_{j k}}$ (and $\smash{\tilde L_{j k}}$) and then deduce the consequences for 
$M_j$ and $\smash{\mathcal{S}\cY_j^{\prime}}$.  

 In the case $S_{j k}$ has no edge curves,  $S_{j k}$ is 
homeomorphic to the space obtained by attaching a $2$--disk to $L_{j k}$ as follows.  
\begin{Lemma}
\label{Lem6.4}
With the preceding notation:
\begin{enumerate}
\item  If $S_{j k}$ has no edge curves, then $S_{j k}$ is homeomorphic to the 
space obtained by attaching a $2$--disk to $L_{j k}$ along the path given in 
\eqref{Eqn6.2}.
\item  If $S_{j k}$ has edge curves, then $L_{j k}$ is a strong deformation 
retract of $S_{j k}$.
\item  $\tilde L_{j k}$ is a strong deformation retract of $L_{j k}$.  
\item  Furthermore, $\smash{L_{j k}}$ is homotopy equivalent (rel $\smash{\gD_{j k}}$) to 
$\smash{\gD_{j k} \vee \bigvee_i S^1}$ (or equivalently $\smash{\tilde L_{j k}}$ is homotopy 
equivalent (rel $\smash{\Gamma_{j k}}$) to $\smash{\Gamma_{j k}  \vee \bigvee_i S^1}$), where 
$\smash{\bigvee_i S^1}$ is a bouquet of $q_{j k}$ $S^1$'s (with $q_{j k}$ defined by equation 
\eqref{Eqn3.2a}).  
\end{enumerate}
\end{Lemma}
\begin{proof}
If $S_{j k}$ has no edge curves then we proceed as follows to recover $S_{j 
k}$ from $L_{j k}$ by attaching a $2$--disk.  We cut the annulus $\tilde S_{j 
k} = S_{j k}\backslash \mbox{int}(\gD_{j k})$ along $\tau_{k 0}$.  We obtain a 
topological $2$--disk.  Its boundary is mapped to $L_{j k}$ by first following 
$\tau_{k 0}$, then the boundary of the $2$--disk $D^2$ in a counter 
clockwise direction, then backwards along $\tau_{k 0}$ and finally around the 
boundary $C_{j k}$ of $\gD_{j k}$ but in a clockwise direction.  We denote the 
counterclockwise path around $\partial D^2$ by $\zeta_{j k}^{\prime}$, and 
the counterclockwise path around the boundary $\gD_{j k}$ by $\gd_{j k}$.  
Then a $2$--disk is attached along its boundary to $L_{j k}$ by the path 
\begin{equation}
\label{Eqn6.2}
 \xi_{j k} \, = \, \zeta_{j k} * \bar \gd_{j k} \quad \mbox{where } \quad 
\zeta_{j k} = \tau_{k 0}* \zeta_{j k}^{\prime} *\bar \tau_{k 0}
\end{equation} followed by the identification (as is common, $\bar \ga$ 
denotes the inverse of a path $\ga$).  

Next, suppose instead $S_{j i}$ has $e_{j k} > 0$ edge curves.  Then we cut 
$\tilde S_{j k} = S_{j k}\backslash \mbox{int}(\gD_{j k})$ along the paths 
$\tau_{k 0}$ and $\tau_{k i}$ for $i < e_{j k}$.  The resulting space is 
topologically an annulus with one boundary the edge curve $\smash{E^{\prime} 
\overset{\text{def}}{=} E_{j e_{j k}}}$ of the $e_{j k}$--th hole.  Then we may 
retract the annulus onto $L_{j k}$ by retracting along curves from the 
boundary edge of $E^{\prime}$ as shown, followed by the identification.  This 
gives $L_{j k}$ as a strong deformation retract of $S_{j k}$.  

Since this fixes $\gD_{j k}$, we may then further retract $\gD_{j k}$ onto 
$\Gamma_{j k}$, while fixing the rest of $L_{j k}$.  This gives $\smash{\tilde L_{j k}}$ 
as a strong deformation retract of $L_{j k}$.  

Lastly, the boundary of $D^2$ after identification is a bouquet of $S^1$\rq 
s, where the number is $2g_{k}$ if $S_{j k}$ is orientable, and $g_{k}$ if 
$S_{j k}$ is nonorientable.  

Finally, returning to $L_{j k}$, if we collapse all of the paths  $\tau_{k 0}$ and 
$\tau_{k i}$, $i < e_{j k}$, to the point $x_{j 0}$, then $\smash{\tsty\bigcup_{i < e_{j k}} E_{j 
i} \cup \tau_j \cup (\partial D^2/\sim )}$ becomes a bouquet of $q_{k j}$ 
$S^1$\rq s.  This collapsing induces a projection map from $\smash{L_{j k}}$ to 
$\smash{\gD_{j k} \vee \bigvee_{i = 1}^{q_{k j}} S^1}$ which is the desired homotopy 
equivalence.  An analogous argument works for the space $\smash{\tilde L_{j k}}$.
\end{proof}
Since the deformation retraction for $L_{j k}$ (resp.\ $\tilde L_{j k}$)  in (3) of \fullref{Lem6.4} fixes $\gD_{j k}$ (resp.\ $\Gamma_{j k}$), we may apply the Lemma to each medial 
sheet with an edge curve after it is attached to $\cY_j$.  If we let $\smash{\mathcal{S}\cY_j^{\prime}}$ 
denote $\mathcal{S}\cY_j$ with all of these medial sheets attached, then even though each medial 
sheet $S_{j k}$ contracts to a bouquet of $S^1$\rq s attached to $\gD_{j k}$, as $\mathcal{S}\cY_j$ 
is path connected, up to homotopy we may bring these bouquets together and conclude the 
following \lq\lq CW--decomposition of $M_j$\rq\rq.  
\begin{Proposition}
\label{Prop6.4}
For an irreducible medial sheet $M_j$, we have the following:
\begin{enumerate}
\item  $M_j$ is homotopy equivalent to a subspace $M_j^{\prime}$ obtained 
from $\mathcal{S}\cY_j^{\prime}$ by attaching for each sheet $S_{j k}$ without 
edge curve a $2$--disk along the closed path given by \eqref{Eqn6.2}.
\item $\mathcal{S}\cY_j^{\prime}$ is homotopy equivalent (rel $\mathcal{S}\cY_j$) to 
$\big(\bigvee_{i = 1}^{q_j} S^1\big) \vee \mathcal{S}\cY_j$, and hence to the $1$--complex 
$\big(\bigvee_{i = 1}^{q_j} S^1\big) \vee \Gamma_j$ (here $q_j$ is defined in \fullref{S:sec3}).  
\end{enumerate}
\end{Proposition}
\begin{proof}
By \fullref{Lem6.4}, we can first strong deformation retract each 
sheet $S_{j k}$ with an edge curve onto $L_{j k}$.  Then we can obtain each 
sheet $S_{j k}$ without an edge curve by attaching a $2$--disk to $L_{j k}$.  
Because this is done while fixing $\mathcal{S}\cY_j^{\prime}$, we may do this all at 
once obtaining the strong deformation retract $M_j^{\prime}$.  

Also, by \fullref{Lem6.4}, each $L_{j k}$ has as a strong deformation 
retract $\tilde L_{j k}$ (rel $\gG_{j k}$), which is homotopy equivalent (rel 
$\Gamma_{j k}$) to $\smash{\big(\bigvee_{i = 1}^{q_{j k}} S^1\big) \vee \Gamma_{j k}}$.  Again 
doing this for each $L_{j k}$ gives $\smash{\big(\bigvee_{i = 1}^{q_j} S^1\big) \vee \Gamma_j}$ 
as a strong deformation retract of $\mathcal{S}\cY_j^{\prime}$.  Thus, $M_j$ is 
homotopy equivalent to the space obtained from $\smash{\big(\bigvee_{i = 1}^{q_j} S^1\big) 
\vee \Gamma_j}$ by attaching one $2$--disk for each medial sheet $S_{j k}$ 
without edge curves.  
\end{proof}

\section{Fundamental group of irreducible medial components}  
\label{S:sec7} To compute the fundamental group of an irreducible medial component $M_j$, we 
use the CW--complex decomposition we have just given in \fullref{S:sec6}.  We successively 
use the form of the fundamental group for $Y\!$--network components given in \fullref{Prop3.1} and that of the graph $\gL_j$ to compute the fundamental group of $\mathcal{S}\cY_j$.  
We then compute the changes obtained by attaching the $L_{j k}$ to give that for 
$\mathcal{S}\cY_j^{\prime}$, and lastly determine the effect of attaching the $2$--disks 
corresponding to $S_{j k}$ without edge curves. 

 We first compute the fundamental group 
of $\mathcal{S}\cY_j$ as follows.
\begin{Proposition}
\label{Prop7.1} 
With the preceding notation,
\begin{equation}
\label{eqn7.1}
 \pi_1(\mathcal{S}\cY_j) \,\, \simeq \,\, \pi_1(\Gamma_j) \,\, \simeq \,\, 
\pi_1(\gL_j) * ( *_{i = 1}^r \pi_1(\cY_{j i}))  
\end{equation}
where there are $r$ components for $\cY_j$. 
\end{Proposition}
\begin{proof}
By the preceding remarks, it is sufficient to compute $\pi_1(\Gamma_j)$, viewing $\Gamma_j$ 
as an extended graph.  We choose a maximal tree $T_{j k}$ for each component $\cY_{j k}$ and 
a maximal tree $R_j$ for $\gL_j$.  Furthermore, as the image of an edge $\partial \bar B_{j 
i}$ will cover completely each edge of $\cY_{j k}$ which it meets, we may choose the points 
$x_{j k}$ so their images lie in $T_{j k}$, eg they are vertices $y_m$.  

We also have to treat the special case where one of the $\cY_{j k}$ is an 
$\circ$ without any vertices.  As in the proof of \fullref{Prop2.5}, 
we choose a point in $\circ$ as an artificial vertex, so the $\circ$ can be 
viewed as a loop on the vertex and then the vertex itself becomes the 
maximal tree.  As earlier, we do not count such \lq\lq vertices\rq\rq\ in 
determining $v_j$ for numerical relations.  

Then we let $\tilde \Gamma_{j k}$ denote the union of the curves $\gb_i$ in 
$\Gamma_{j k}$ and form 
$$  \tilde \Gamma_j \,\, = \,\, (\tsty\bigcup_k \tilde \Gamma_{j k} \cup_i T_{j i})/ 
\sim.  $$
Then each tree $T_{j k}$ is contractible and if we collapse each $T_{j k}$ to 
a point then the quotient is isomorphic to the extended graph $\gL_j$.   

We may lift the edges of the tree $R_j$ to a collection of edges in $\tilde 
\Gamma_j$, by which we mean we can choose curves $\gb_i$ in the $\tilde 
\Gamma_{j k}$ which correspond to the edges of $R_j$.  We denote the union 
of those curves by $R^{\prime}_j$.  Then we form 
$$ \cT_j  \,\, = \,\, (R^{\prime}_j \cup (\tsty\bigcup_i T_{j i}))/ \sim  $$
where now \lq\lq $\sim$\rq\rq\ only involves attaching the points $x_{j k}$ of 
$R^{\prime}_j$ to the corresponding points $y_m$ of the $T_{j i}$.  
 
\begin{Lemma}
\label{Lem7.2}
$\cT_j$ is a maximal tree in the extended graph $\tilde \Gamma_j$.  
\end{Lemma}
\begin{proof}
First, $\cT_j$ is a tree because each $T_{j i}$ is contractible and if we 
collapse each to a point we obtain $R_j$ which is contractible; hence, $\cT_j$ 
is contractible and is a graph.  

To see $\cT_j$ is maximal in $\tilde \Gamma_j$, we note any other edge of 
$\tilde \Gamma_j$ either corresponds to an edge of one of the $\cY_{j k}$, 
or to another edge of $\gL_j$ under the collapsing map.  In the first case, if 
the edge $\eta$ is in $\cY_{j k}$, then as $T_{j i}$ is maximal $T_{j i} \cup 
\eta$ is no longer simply connected; thus nor will be $\cT_j \cup \eta$.  If 
instead the added edge $\eta$ maps to another edge $\eta^{\prime}$ of 
$\gL_j$ under the collapsing map, then $R_j \cup \eta^{\prime}$ is no longer 
simply connected.  Then we may lift a noncontractible loop $\gamma$ of 
$R_j \cup \eta^{\prime}$ to a loop in $\cT_j \cup \eta$ (joining edges of 
$\gamma$ lift to edges which share a common $Y\!$--node, say $\cY_{j k}$, 
and hence can be connected by a path in $T_{j i}$), destroying the simple 
connectivity of $\cT_j$.  Thus, $\cT_j$ is maximal.  
\end{proof}

Then by \fullref{Prop2.5},  $\pi_1(\mathcal{S}\cY_j)$ is a free group with a 
free generator for each edge of $\tilde \Gamma_j$ not in $\cT_j$.  There 
are $\gl_j$ edges of $\gL_j$, and $v_{j k} + 1$ edges for each $\cY_{j k}$ by 
\fullref{Prop3.1}.  Thus, there are a total of $\gl_j + v_j + c_j$ free 
generators (recall $v_j = \sum_k v_{j k}$ and $c_j$ is the number of 
connected components $\cY_{j k}$ of $\cY_j$).  Such a free group is the 
free product of the free groups on $v_{j k} + 1$ generators which are 
isomorphic to $\pi_1(\cY_{j k})$, and a free group on $\gl_j$ generators 
which is isomorphic to $\pi_1(\gL_j)$.  
\end{proof}

The $\smash{\pi_1(\cY_{j k})}$ are naturally identified as subgroups of 
$\smash{\pi_1(\mathcal{S}\cY_j)}$.  Also, 
the collapsing map $\smash{\varphi \co \tilde \Gamma_j \to \gL_j}$ is a homotopy 
equivalence; hence, $\smash{\pi_1(\tilde \Gamma_j) \simeq \pi_1(\gL_j)}$.  Then 
the inclusion $\smash{\tilde \Gamma_j \subset \mathcal{S}\cY_j}$, identifies a subgroup 
isomorphic to $\pi_1(\gL_j)$.

The second step is to compute the fundamental group of $\mathcal{S}\cY_j^{\prime}$.  It is homotopy 
equivalent to $\smash{\big(\bigvee_{i = 1}^{q_j} S^1\big) \vee \mathcal{S}\cY_j}$.  Hence, again by the Seifert--Van Kampen 
theorem, 
$$  \pi_1(\mathcal{S}\cY_j^{\prime}) \,\, \simeq  \,\,  \pi_1(\mathcal{S}\cY_j) * F_{q_j}   
$$
where $F_{q_j}$ is a free group on $q_j$ generators corresponding to the 
generators of $\smash{\bigvee_{i = 1}^{q_j}\! S^1}$ obtained from the $E_{j k}$ and the 
identifications on the boundaries $\partial D^2/ \sim$.  Hence,

\begin{Corollary}
\label{Cor7.3}
$\pi_1(\mathcal{S}\cY_j^{\prime})$ is a free group on $Q_j (= \gl_j + v_j + c_j + 
q_j)$ generators.
\end{Corollary}
\subsection*{Computing the fundamental group of an irreducible medial 
component}  We now use the proceeding to compute the fundamental group of $M_j$ by 
determining the effect of attaching each medial sheet $S_{j k}$.  

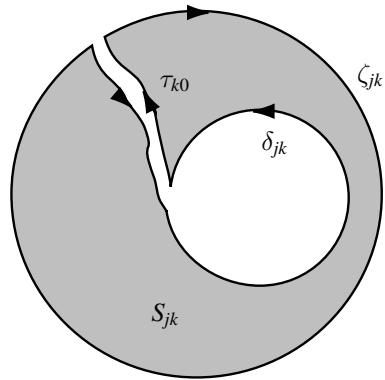
\begin{figure}[ht!]
\labellist
\small\hair 2pt
\pinlabel {$\tau_{k0}$} at 190 690
\pinlabel {$S_{jk}$} at 186 592
\pinlabel {$\delta_{jk}$} [t] at 232 672
\pinlabel {$\zeta_{jk}$} [bl] at 265 686
\endlabellist
\centerline{\includegraphics[width=5cm]{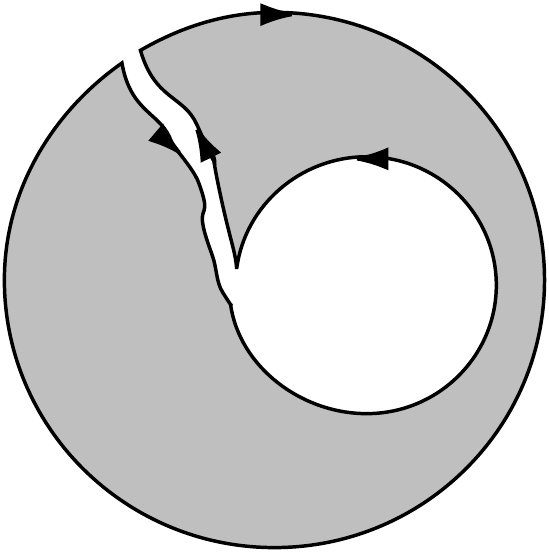}}
\caption{\label{fig.11a} Representing  $S_{j k}$ by attaching a $2$--disk to  
$\mathcal{S}\cY_j^{\prime}$}
\end{figure} 

Then we can compute the fundamental group of $M_j$ by attaching the 
remaining $2$--disks to $\smash{\mathcal{S}\cY_j^{\prime}}$; see \fullref{fig.11a}. 
First, we recall the loops  at $x_{k 0}$: $\smash{\zeta_{j k}}$ for each medial sheet 
without edge curves, and $\smash{\gd_{j k}}$ (obtained by following $C_{j k}$ in the 
counterclockwise direction).  We again let $\cT_j$ denote a maximal tree in 
$\smash{\tilde \Gamma_j}$ and we choose a base point $x_0 \in \smash{\tilde \Gamma_j}$, 
and for each $x_{k 0}$, a path $\ga_k$ in the tree $\cT_j$ from $x_0$ to 
$x_{k 0}$.  As $\cT_j$ is a tree, the path homotopy class of $\ga_k$ only 
depends on $x_{k 0}$.  We form 
\begin{equation*}
 \tilde \gd_{j k} = \bar \ga_k * \gd_{j k}* \ga_k \quad \mbox{ and } \quad 
\tilde \zeta_{j k} = \bar \ga_k * \zeta_{j k} * \ga_k.
\end{equation*}  
These define elements $\smash{r_{j k} = \tilde \gd_{j k} * \tilde \zeta_{j k}^{-1}}$ in 
$\pi_1(\mathcal{S}\cY_j^{\prime}, x_0)$, which after composing with the inclusion map yields elements 
of $\pi_1(M_j, x_0)$.   

\begin{Thm}
\label{Thm7.5}
$\pi_1(M_j)$ has the representation 
\begin{equation}
\label{Eqn7.5a}
    \pi_1(M_j)   \quad \simeq \quad F_{Q_j}/N   
\end{equation}
where $F_{Q_j}$  is a free group on $Q_j$ generators, and $N$ is the normal 
subgroup generated by the elements $\smash{r_{j k} = \tilde \zeta_{j k} * (\tilde 
\gd_{j k} )^{-1}}$, one for each medial sheet without edge curves.  
Equivalently,  $\pi_1(M_j)$ has exactly the relations 
\begin{equation}
\label{Eqn7.5b}
r_{j k} = \tilde \gd_{j k} * \tilde \zeta_{j k}^{-1} = 1 \quad \mbox{ or 
equivalently }  \quad \tilde \gd_{j k} \,\, = \,\,  \tilde \zeta_{j k}.    
\end{equation}
Also, for $F_{Q_j}$, $q_j$ of the generators correspond to the generators from the identified 
boundaries of the $2$--disks, and $v_j + c_j$ generators come from $\cY_j$.  

Also, $\gl_j$ generators come from the generators for $\pi_1(\gL_j)$.  In fact, the 
collapsing map $\varphi_j \co \mathcal{S}\cY_j \to \gL_j$ extends to a continuous map $\tilde \varphi_j 
\co M_j \to \gL_j$ which is a surjective map on $\pi_1$.
\end{Thm}
\begin{proof}
$M_j$ is obtained up to homotopy equivalence by attaching $2$--disks to 
$\mathcal{S}\cY_j^{\prime}$ along the closed paths given by \eqref{Eqn6.3}; see \fullref{fig.11a}.  By the Seifert--Van Kampen theorem this introduces relations 
killing the corresponding homotopy classes given by \eqref{Eqn6.3}, but with 
the base point moved to $y_j$. These are exactly the relations given in  
\eqref{Eqn7.5b}.  The identifications of the various generators from the 
$\cY_{j k}$, the edge curves of the $S_{j k}$, and the identified boundaries 
of the $2$--disks follow from the Seifert--Van Kampen theorem.  

To show that the collapsing map $\varphi_j \co \mathcal{S}\cY_j \to
\gL_j$ extends to $M_j$, we first construct a quotient $\tilde M_j$ of
$M_j$ by collapsing the components $\cY_{j k}$ to points.  Within
$\tilde M_j$ is the subspace $\tilde{\mathcal{S}}\cY_j$, formed from
$\mathcal{S}\cY_j$ by collapsing the components $\cY_{j k}$ to points.
The collapsing map $\varphi_j \co \mathcal{S}\cY_j \to \gL_j$ factors
through $\tilde{\mathcal{S}}\cY_j$.  Hence, it is sufficient to extend
this collapsing map to $\tilde M_j$ and then compose with the quotient
map $M_j \to \tilde M_j$.

Then in $\tilde{\mathcal{S}}\cY_j$, $\gD_{j k}$ has its inner boundary edges collapsed 
to points to form $\tilde \gD_{j k}$, which is topologically a $2$--disk.  As a 
$2$--disk is an absolute retract, we can extend the identity map on $\tilde 
\gD_{j k}$ to a continuous map $\tilde S_{j k} \to \tilde \gD_{j k}$, where 
$\tilde S_{j k}$ is $S_{j k}$ with each boundary edge in $\gD_{j k}$ collapsed 
to a point.  Hence, these continuous maps together give a retract $\tilde M_j 
\to \tilde{\mathcal{S}}\cY_j$.  Composed with the collapsing map gives the continuous 
extension $\tilde \varphi_j$.  

Finally the collapsing map $\tilde \Gamma_j \to \gL_j$ is a homotopy 
equivalence, and factors through $\tilde \varphi_j \co M_j \to \gL_j$.  Thus, 
the induced map on $\tilde \varphi_{j *} \co \pi_1(M_j) \to \pi_1(\gL_j)$ is 
surjective.
\end{proof}

Hence, in the special case when $M_j$ is simply connected, we obtain \fullref{Thm5.0}.

\section{Exact sequence for attaching sheets}  
\label{S:sec8}  Next, we compute the homology and Euler characteristic of an irreducible 
medial component $M_j$. We use the two ways of obtaining $M$: by the medial sheet 
representation, attaching surfaces with boundaries to $\mathcal{S}\cY_j$, or the CW-representation 
obtained by attaching cells to $\mathcal{S}\cY_j^{\prime}$.  Using these we may compute the exact 
sequences of pairs $(M_j, \mathcal{S}\cY_j)$ and $(M_j, \mathcal{S}\cY_j^{\prime})$.  Furthermore, we have 
the inclusion of pairs $(M_j, \mathcal{S}\cY_j) \subset (M_j, \mathcal{S}\cY_j^{\prime})$  This leads to the 
homomorphism of exact sequences of pairs using reduced homology.  Since all spaces have 
dimension at most $2$ the sequence ends at $H_2$.  
\begin{equation}
\begin{split}
\label{CD8.0}
&\begin{CD} 
{0 } @>>> {H_2(\mathcal{S}\cY_j)} @>>>  { H_2(M_j) } @>>> {H_2(M_j, \mathcal{S}\cY_j)}    
@>\gd>>  \\
 @|    @VVV     @|      @VVV   \\
{0 }  @>>>  {H_2(\mathcal{S}\cY_j^{\prime})}  @>>>  { H_2(M_j) }  @>>>  {H_2(M_j, 
\mathcal{S}\cY_j^{\prime})}  @>\gd>>  
\end{CD}   \\
\vspace{2ex}
&\hspace{6em}\begin{CD}\vrule width 0pt height 15pt  {H_1(\mathcal{S}\cY_j)} @>>>  { H_1(M_j) } @>>> {H_1(M_j, \mathcal{S}\cY_j)} 
@>>> { 0}  \\
   @VVV   @|     @VVV    @|    \\
{H_1(\mathcal{S}\cY_j^{\prime})}  @>>>  { H_1(M_j) }  @>>>  {H_1(M_j, 
\mathcal{S}\cY_j^{\prime})}  @>>>  { 0}
\end{CD} 
\end{split}
\end{equation}

First, using \fullref{Prop6.1} and \fullref{Prop6.4}  we can compute 
the groups $H_i(\mathcal{S}\cY_j)$ and $H_i(\mathcal{S}\cY_j^{\prime})$, which are zero 
for $i > 1$  and for $i = 1$
\begin{gather*}
H_1(\mathcal{S}\cY_j) \,\, \simeq \,\, \Z^{\gl_j + v_j + c_j}\\
\tag*{\hbox{and}} 
H_1(\mathcal{S}\cY_j^{\prime}) \,\, \simeq \,\, \Z^{\gl_j + v_j + c_j + q_j}  \,\, 
\simeq \,\, \Z^{Q_j}.
\end{gather*}
Second, we can compute the relative groups by excision using \fullref{Prop6.1} and \fullref{Prop6.4}:
\begin{equation}
\label{Eqn8.3a}
H_i(M_j, \mathcal{S}\cY_j) \,\, \simeq \,\, \oplus_k H_i(\tilde S_{j k}, C_{j k})
\end{equation}
where the sum is over all medial sheets, and 
\begin{equation*}
H_i(M_j, \mathcal{S}\cY_j^{\prime}) \,\, \simeq \,\, \oplus_k H_i(D^2, \partial D^2)
\end{equation*}
where the second sum is over medial sheets without edge curves.
From this, we conclude $H_i(M_j, \mathcal{S}\cY_j^{\prime}) = 0$ 
if $i \neq 2$, and 
\begin{equation*}
H_2(M_j, \mathcal{S}\cY_j^{\prime}) \,\, \simeq \,\,  \Z^{s_{0 j}}
\end{equation*}
(recall $s_{0 j}$ denotes the number of medial sheets of $M_j$ which are 
without edge curves).  Likewise, from \eqref{Eqn8.3a}, we compute $H_i(M_j, 
\mathcal{S}\cY_j)$
using the next lemma.
\begin{Lemma}
\label{Lem8.5}
Suppose $N$ is a compact surface with boundary, with a distinguished 
boundary component $C$.  Let $N^{\prime}$ denote the surface obtained 
from $N$ by attaching a $2$--disk to $N$ along the boundary component $C$.  
Then
$$  H_i(N, C) \,\, \simeq \,\,  \tilde H_i(N^{\prime}).       $$
\end{Lemma}
\begin{proof}
By excision
$$   H_i(N^{\prime}, D^2)  \,\, \simeq \,\,  H_i(N, C).    $$
Then by the exact sequence of the pair $(N^{\prime}, D^2)$ using reduced 
homology, we conclude
$$   \tilde H_i(N^{\prime})  \,\, \simeq \,\,  H_i(\tilde N, D^2).    \proved$$
\end{proof}
Hence, by \fullref{Lem8.5}: 
\begin{enumerate}
\item If $N$ has more than one boundary component, or $N$ is nonorientable, 
then $H_2(N, C) = 0$; while if $N$ is orientable with only the single boundary 
component $C$, then $H_2(N, C) = \Z$.  
\item For $(N, C) = (\tilde S_{j k}, C_{j k})$, $\mbox{rk}(H_1(\tilde S_{j k}, 
C_{j k})) = q_{j k} - \gevar_{j k}$, where $\gevar_{j k} = 1$ if $e_{j k} = 0$ 
and $\tilde S_{j k}$ is nonorientable, otherwise  $\gevar_{j k}= 0$.
\end{enumerate}
Thus,  by \eqref{Eqn8.3a}, $H_i(M_j, \mathcal{S}\cY_j) = 0$ for $i > 2$, and 
\begin{equation*}
H_i(M_j, \mathcal{S}\cY_j) \,\, \simeq \,\, \Bigg\{ 
\begin{array}{ll} 
{\Z^{s_{0 o, j}}}  &  {i = 2}  \\ 
{\Z^{(q_j - s_{0 n, j})} \oplus (\Z_2)^{s_{0 n, j}} } &  {i = 1}  
\end{array}  
\end{equation*}
where we recall $s_{0 o, j}$, resp.\ $s_{0 n, j}$, denotes the number of orientable, resp.\ 
nonorientable, medial sheets with no edge curves.  

Hence, the homomorphism of exact sequences \eqref{CD8.0} takes the form
\begin{equation}
\begin{split}
\label{CD8.8}
\begin{CD} 
{0 } @>>> {0 }  @>>>  { H_2(M_j)}  @>>> {\Z^{s_{0 o, j}} }    @>\gd>>  \\
 @|      @|     @|      @V{\iti}VV   \\
{0 } @>>> {0 } @>>>  { H_2(M_j) }  @>>> {\Z^{s_{0 j}} }  @>\gd>>  
\end{CD}\hspace{1.4in}\\
\vspace {2ex}
\hfill  \begin{CD}\vrule width0pt height 15pt  {\Z^{(\gl_j + v_j + c_j)} } @>>>  { H_1(M_j)}  @>>> 
{\Z^{(q_j - s_{0 n, j})} \oplus (\Z_2)^{s_{0 n, j}} } @>>> { 0}  \\
   @ViVV   @|     @V0VV    @|    \\
{\Z^{(\gl_j + v_j + c_j + q_j)} } @>>>  { H_1(M_j)} @>>> {0 } @>>> { 0}
\end{CD} 
\end{split}
\end{equation}

Recall $\Z^{s_{0 j}} = \Z^{s_{0 n, j}} \oplus \Z^{s_{0 o, j}}$ is the number of 
medial sheets without medial edge curves, and $\iti$ denotes inclusion of 
$\Z^{s_{0 o, j}}$.  

We can draw a number of conclusions from the diagram \eqref{CD8.8}.
 We let $b_i$ denote the $i$--th Betti number of $M_j$, and we can take the 
Euler characteristics of either row and obtain the same formula.  For 
example, from the first row of \eqref{CD8.8},
\begin{equation*}
b_2 - s_{0 o, j} + (\gl_j + v_j + c_j) - b_1 + (q_j - s_{0 n, j}) \quad = \quad 0. 
\end{equation*}
This yields
\begin{equation*}
\tilde \chi(M_j)  \quad  = \quad s_{0 j} - (\gl_j + v_j + c_j + q_j)  \quad = 
\quad s_{0 j} - Q_j.
\end{equation*}
Substituting \eqref{Eqn3.2c} into this equation yields  
\eqref{Eqn3.3b} as asserted.  

\subsection*{Computing homology by the algebraic attaching 
homomorphism} Next, we use the second row.  The middle homomorphism is exactly the algebraic 
attaching homomorphism
\begin{equation}
\label{Eqn8.11}
\Psi_j \co \Z^{s_{0 j}}  \overset{\gd}{\to}  \Z^{Q_j}. 
\end{equation}
It sends the generator of $H_2(\tilde S_{j k}, C_{j k})$ for orientable  $S_{j k}$ without 
edge curves to the image of its boundary (ie $C_{j k}$) under the attaching map to $\cY_j$.  
Here $\Z^{Q_j}$ represents the first homology of $\mathcal{S}\cY_j^{\prime}$.  Then by the exactness 
of the bottom row of \eqref{CD8.8}, the kernel and cokernel of the $\Psi_j$ are $H_2(M_j)$, 
respectively $H_1(M_j)$, yielding (2) of \fullref{Thm3.3}.  Furthermore, by the top row 
of \eqref{CD8.8}, we obtain $\mbox{rk} (H_2(M_j; \Z)) \leq s_{0 o, j}$.  As $H_*(M_j; \Z)$ 
is torsion free, we can tensor the top row with $\Z/2\Z$ and obtain $\mbox{rk} (H_1(M_j; \Z)) 
\geq (q_j - s_{0 n, j}) + s_{0 n, j} = q_j$.  Also,  by the bottom row, or the formula for 
$\tilde \chi (M_j)$, 
$$\mbox{rk} (H_1(M_j; \Z))  \,  = \, Q_j - s_{0 j} + b_2  \, \leq \,  Q_j - 
s_{0 j} + s_{0 o, j} \, =  \,  Q_j - s_{0 n, j}. $$
This completes the proof of \fullref{Thm3.3}

The third consequence of diagram \eqref{CD8.8} is for the cases when 
various homologies of $M_j$ vanish.  
\begin{proof}[Proof of \fullref{Cor3.4}]
First, suppose $H_1(M_j) = 0$.  From the first row, we conclude that both 
$q_j - s_{0 n, j} = 0$ and $s_{0 n, j} = 0$.  Hence, $q_j = s_{0 n, j} = 0$.  
Then $q_j = \sum_k q_{j k}$ is a sum of nonnegative integers, so $q_{j k} = 
0$.  By \eqref{Eqn3.2a}, each $g_{j k} = 0$ and if $e_{j k} > 0$, then $e_{j k} 
= 1$.  In particular, $S_{j k}$ must be a $2$--disk with a finite number of 
holes, and the edge of at most one of the holes is an edge curve of $M_j$.  
In particular, all of the medial sheets are orientable so 
\begin{equation}
\label{Eqn8.12}
s_{0 o, j} \,\, = \,\, s_{0 j} \,\, = \,\, s_j \, - \, e_j.
\end{equation}  
 This establishes (2) and (3) of \fullref{Cor3.4}.  Furthermore, by 
\fullref{Thm7.5}, the collapsing map $p \co M_j \to \gL_j$ induces a 
surjection on $\pi_1$ and hence on the abelianized $\pi_1$.  Hence, 
$H_1(\gL_j; \Z) = 0$, so $\gL_j$ is a tree and $\gl_j = 0$, establishing (1) of 
\fullref{Cor3.4}. 

 Second, suppose in addition that $H_2(M_j; \Z) = 0$.  Then
$$  \gd \co \Z^{s_{0 o, j}}  \simeq   \Z^{(\gl_j + v_j + c_j)}.  $$ 
$$
s_{0 o, j} \,\, = \,\,\gl_j \, + \, v_j \, + \, c_j.
\leqno{\hbox{Thus,}}
$$ 
 Hence, as $\gl_j = 0$, this equation and \eqref{Eqn8.12} together imply
\begin{equation*}
s_j \, - \, e_j \,\, = \,\, v_j \, + \, c_j.
\end{equation*}  
    The remaining conditions on the fundamental group follow from 
\fullref{Thm7.5}.  
\end{proof}

\section{Algorithm for contracting contractible medial axes}  
\label{S:sec9} In the contractible case, the structure suggests a method for contracting a 
region.  We would like a specific algorithm for contracting the medial axis when it is 
contractible.  First, we can deform it by sliding along fin curves so we have the simplified 
structure $\hat M$.  Then we want to simplify each irreducible medial component.  If the 
medial component has a nonempty $Y\!$--network , then we choose a medial sheet with a medial 
edge curve, and deform the medial sheet so the medial edge curve touches the $Y\!$--network.  
At that point a transition occurs with two fin points being created as in \fullref{fig.15}.  
We can then slide the fin curve until it doesn\rq t meet the remaining $Y\!$--network.  This 
does not change the homotopy type of the medial component.  

\begin{figure}[ht!]
\centerline{\includegraphics[width=10cm]{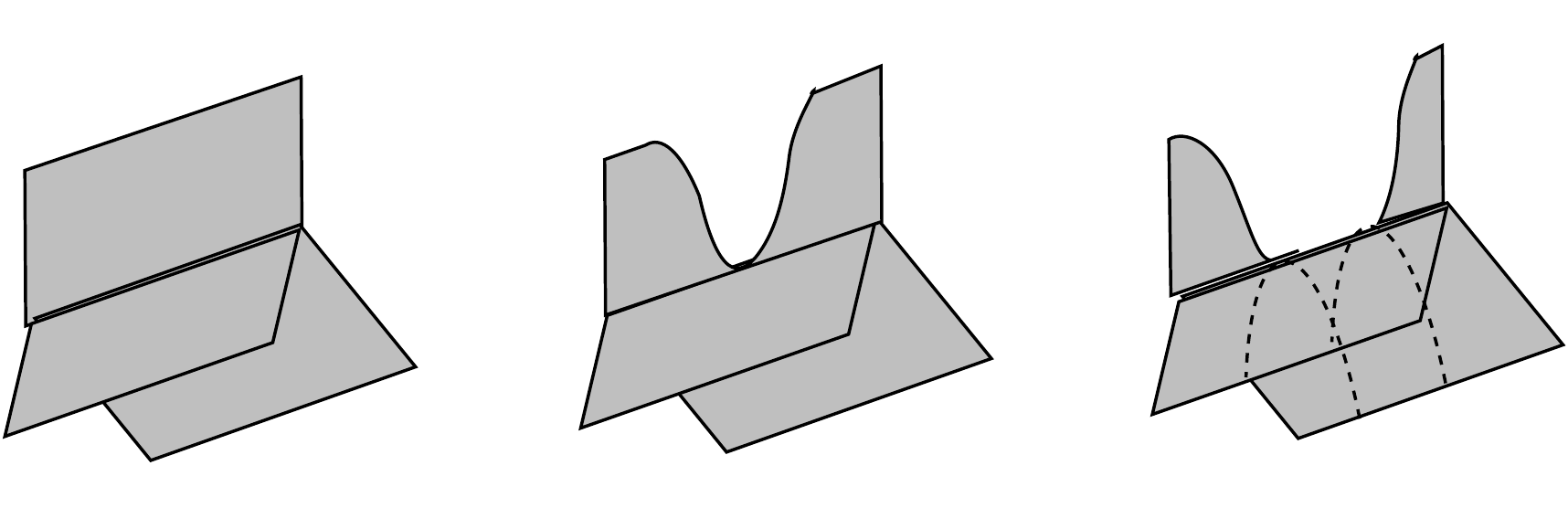}}
\caption{\label{fig.15} Creation of a pair of fin points from a medial edge 
curve}
\end{figure} 

We repeat this process as long as the $Y\!$--network of the component is 
nonempty.  Applying this to all components, we reach a point where all 
$Y\!$--networks are empty.  In the process we have created more irreducible 
components attached to each other along fin curves.  By contractibility, all 
of the medial sheets are $2$--disks.  The attaching is described by a tree.  
We then proceed to simplify the tree.  

We take one of the outermost branches.  It corresponds to a $2$--disk
attached to another sheet along a segment of its boundary as a fin
curve.  We contract down the sheet and the fin curve to a point.  Then
we repeat this process until there is only a single medial component
left.  It is a $2$--disk and hence contractible.

The only point about the preceding argument is the assertion that if a medial 
component is contractible then it has an edge curve.  This step has not been 
completed and we are left with a conjecture.  

\medskip
{\bf Conjecture}\qua If an \lq\lq irreducible medial component\rq\rq\ is contractible then it 
has an edge curve. \medskip

The \lq\lq irreducible medial component\rq\rq\ is in quotes because after the 
first step we no longer know that it really comes from the medial axis of a 
region.  What we really are asking is whether a singular space which is 
embedded in $\R^3$ and has the local singular structure of a medial axis 
satisfies the conjecture. 

\section{Discussion and open questions}  
\label{S:sec10}

There are a number of basic questions which are unanswered.
\begin{enumerate}
\item  The conditions we gave are for a medial axis which already exists.  
Given an abstract model for a medial axis, satisfying all of the properties, this model may 
not be realized in $\R^3$.  In particular, we must understand the conditions allowing 
embeddings of the medial sheets  so they do not intersect.   
\item  Likewise, there are related restrictions on the attaching maps to 
$Y\!$--network so that three sheets are attached at each point.  What 
constraints does this place on possible structures?  This and the preceding 
point appear to involve subtle questions in topology. 
\item In order to complete the algorithm for contracting regions, we need to 
prove the existence of edge curves on contractible irreducible \lq\lq medial 
components\rq\rq.  
\item   The decomposition into irreducible medial components requires 
operations on fin curves.  These local operations (with the exception of the 
cutting of essential fin curves) are examples of local deformations of medial 
axes resulting from deformations of regions by results of Giblin and Kimia 
\cite{GK} building on the results of Bogaevsky \cite{Bg,Bg2}.  Can we 
obtain the operations on fin curves as global deformations of the region? 
\item  The structure of the medial axis for specific types of regions such 
as, for example, knot complement regions contains topological information 
about the region.  For example, if we look for a topological representative of 
the region for which the number of medial sheets is minimal, then as the 
reduced Euler characteristic is fixed, we obtain from \eqref{Eqn3.3b} an 
expression for $s_j - \tilde \chi(M_j)$ as a sum of five nonnegative  integers 
$e_j +  v_j +  c_j +  G_j + \gl_j$.  What possible partitions of $s_j - \tilde 
\chi(M_j) $ into five nonnegative integers are actually possible, and how does 
this relate to the structure of the original knot?
\end{enumerate}

\bibliographystyle{gtart}
\bibliography{link}

\end{document}